\documentclass[reqno]{amsart}
\usepackage{amsmath, amssymb, amsthm, epsfig}
\usepackage{hyperref, latexsym}
\usepackage{url}
\usepackage[mathscr]{euscript}

\usepackage{color}
\usepackage{fullpage} 
\usepackage{setspace}

\newcommand{\new}[1]{\textcolor{black}{#1}}

\def\today{\ifcase\month\or
  January\or February\or March\or April\or May\or June\or
  July\or August\or September\or October\or November\or December\fi
  \space\number\day, \number\year}
\DeclareMathOperator{\sgn}{\mathrm{sgn}}

 \newtheorem{theorem}{Theorem}
  
 \newtheorem{lemma}[theorem]{Lemma}
 \newtheorem{proposition}[theorem]{Proposition}
 \newtheorem{corollary}[theorem]{Corollary}
 \theoremstyle{definition}

 \theoremstyle{remark}

 \newcommand{\mc}{\mathcal}

 \newcommand{\J}{\mc{J}}

 \newcommand{\U}{\mc{U}}

 \newcommand{\C}{\mathbb{C}}
 \newcommand{\R}{\mathbb{R}}
 \newcommand{\N}{\mathbb{N}}
 
 \newcommand{\Z}{\mathbb{Z}}

 \newcommand{\hh}{\tfrac12}

 \newcommand{\dt}{\text{\rm d}t}
 \newcommand{\du}{\text{\rm d}u}

 \newcommand{\dx}{\text{\rm d}x}

\newcommand{\dalpha}{\text{\rm d}\alpha}

 \newcommand{\dmu}{\text{\rm d}\mu}

\newcommand{\sumstar}{\sideset{}{^\star}\sum}

\newcommand{\hP}{\widetilde{\Phi}}
\newcommand{\ov}{\overline}
\renewcommand{\H}{\mc{H}}
\newcommand{\im}{{\rm Im}\,}
\newcommand{\re}{{\rm Re}\,}

\begin{document}

\title[Pair correlation of zeros of the Riemann zeta-function]{Hilbert spaces and the pair correlation \\ of zeros of the Riemann zeta-function}
\author[Carneiro, Chandee, Littmann and Milinovich]{Emanuel Carneiro, Vorrapan Chandee, Friedrich Littmann and Micah B. Milinovich}
\date{\today}
\address{IMPA - Instituto de Matem\'{a}tica Pura e Aplicada - Estrada Dona Castorina, 110, Rio de Janeiro, RJ, Brazil 22460-320}
\email{carneiro@impa.br}
\address{Department of Mathematics, Burapha University, 169 Long-Hard Bangsaen Road, Saen Sook Sub-district, Mueang District, Chonburi 20131, Thailand}
\email{vorrapan@buu.ac.th}
\address{Department of mathematics, North Dakota State University, Fargo, ND 58105-5075 USA}
\email{friedrich.littmann@ndsu.edu}
\address{Department of Mathematics, University of Mississippi, University, MS 38677 USA}
\email{mbmilino@olemiss.edu}

\allowdisplaybreaks
\numberwithin{equation}{section}

\maketitle


\begin{abstract}
Montgomery's pair correlation conjecture predicts the asymptotic behavior of the function $N(T,\beta)$ defined to be the number of pairs $\gamma$ and $\gamma'$ of ordinates of nontrivial zeros of the Riemann zeta-function satisfying $0<\gamma,\gamma'\leq T$ and $0 < \gamma'-\gamma \leq 2\pi \beta/\log T$ as $T\to \infty$. In this paper, assuming the Riemann hypothesis, we prove upper and lower bounds for $N(T,\beta)$, for all $\beta >0$, using Montgomery's formula and some extremal functions of exponential type. These functions are optimal in the sense that they majorize and minorize the characteristic function of the interval $[-\beta, \beta]$ in a way to minimize the $L^1\big(\R, \big\{1 -   \big(\frac{\sin \pi x}{\pi x}\big)^2 \big\}\,\dx\big)$-error. We give a complete solution for this extremal problem using the framework of reproducing kernel Hilbert spaces of entire functions. This extends previous work by P. X. Gallagher \cite{G} in 1985, where the case $\beta \in \frac12 \N$ was considered using non-extremal majorants and minorants.
\end{abstract}


\section{Introduction}

 Let $\zeta(s)$ denote the Riemann zeta-function. Understanding the distribution of the zeros of $\zeta(s)$ is an important problem in number theory. In this paper, assuming the Riemann hypothesis (RH), we study the pair correlation function
\begin{equation*}\label{N alpha}
N(T,\beta) \, := \!\!\! \sum_{\substack{ 0< \gamma,\gamma' \le T \\ 0< \gamma'-\gamma \le \frac{2\pi \beta}{\log T} } } \!\!\!\!\! 1\,,
\end{equation*}
where the sum runs over two sets of nontrivial zeros $\rho=\frac{1}{2}+i\gamma$ and $\rho'=\frac{1}{2}+i\gamma'$ of $\zeta(s)$. Here and throughout the text, all sums involving the zeros of $\zeta(s)$ are counted with multiplicity. The pair correlation conjecture of H. L. Montgomery \cite{M1} asserts that
\begin{equation}\label{PCC}
N(T,\beta) \, \sim \, N(T) \int_0^\beta \left\{1 -\Big(\frac{\sin \pi x}{\pi x}\Big)^2 \right\}  \,\dx
\end{equation}
for any fixed $\beta>0$ as $T\to \infty$, where $N(T)$ denotes the number of nontrivial zeros of $\zeta(s)$ with ordinates $\gamma$ satisfying $0<\gamma\le T$. It is known that
\begin{equation}\label{N}
N(T) \, :=\sum_{0<\gamma \le T} 1 \, \sim \, \frac{T \log T }{2\pi}
\end{equation}
as $T\to \infty$. Therefore, if we let $0 < \gamma_1 \le \gamma_2 \le \ldots $ denote the sequence of ordinates of nontrivial zeros of $\zeta(s)$ in the upper half-plane, it follows that average size of $\gamma_{n+1}-\gamma_n$ is about $2\pi/\log \gamma_n$. Thus, the quantity $N(T,\beta)$ essentially counts the number of pairs $0<\gamma,\gamma'\leq T$ of (not necessarily consecutive) ordinates of nontrivial zeros of $\zeta(s)$ whose difference is less than or equal to $\beta$ times the average spacing. It is known that the function $N(T,\beta)$ is connected to the distribution of primes in short intervals, see \cite{GaMu,Gold2,GM}.

\smallskip

Montgomery's pair correlation conjecture is a special case of the more general conjecture that the normalized spacings between the ordinates of the nontrivial zeros of $\zeta(s)$ follow the GUE distribution from random matrix theory. In his original paper \cite{M1}, Montgomery gave some theoretical evidence for the pair correlation conjecture, and later, Odlyzko  \cite{O} provided numerical evidence. Higher correlations of the zeros of $\zeta(s)$, and of the zeros of more general $L$-functions, were studied by Hejhal \cite{H} and by Rudnick and Sarnak \cite{RS2}.

\smallskip

If the asymptotic formula in \eqref{PCC} remains valid when $\beta=\beta(T) \to \infty$ (sufficiently slowly) as $T\to\infty$, one should expect
\begin{equation}\label{PCC2}
N(T,\beta) \sim N(T) \left\{ \beta-\frac{1}{2}+\frac{1}{2\pi^2 \beta} + O\left(\frac{1}{\beta^2}\right) \right\}
\end{equation}
as $T\to\infty$, where the implied constant is independent of $\beta$.  Using techniques of Selberg, Fujii \cite{F2} proved the unconditional estimate
\begin{equation}\label{Fujii}
N(T,\beta) = N(T) \, \big\{ \beta + O(1) \big\}
\end{equation}
for $\beta =O(\log T)$. This improved upon an earlier result of Mueller (unpublished but announced in \cite{G}). 

\smallskip

\subsection{Montgomery's formula and bounds for the pair correlation} For our purposes we define a class of {\it admissible functions} consisting of all $R \in L^1(\R)$ whose Fourier transform  
\begin{equation*} 
\widehat{R}(t) = \int_{-\infty}^{\infty}e^{-2\pi i x t}\,R(x)\, \dx
\end{equation*}
is supported in $[-1,1]$. By the Paley-Wiener theorem, this class of admissible functions is exactly the class of entire functions of exponential type\footnote{An entire function $g: \mathbb C \rightarrow \mathbb C$ has exponential type at most $2\pi\Delta$ if, for all $\epsilon > 0$, there exists a positive constant $C_\epsilon$ such that $|g(z)| \leq C_{\epsilon}\,e^{(2\pi\Delta + \epsilon)z}$ for all $z \in \C$.} at most $2\pi$ whose restriction to the real axis is integrable. An important tool in the study of the correlation of zeros of $\zeta(s)$ is {\it Montgomery's formula}\footnote{This is not Montgomery's original version of his formula. For a derivation of \eqref{Mont_formula}, see the appendix of \cite{G} or \S2.1 below.}, which asserts that, for an admissible function $R$, under RH,  we have
\begin{align}\label{Mont_formula} 
\begin{split}
\lim_{T \rightarrow \infty}\frac{1}{N(T)}  \sum_{0 < \gamma, \gamma' \leq T} & R\!\left( (\gamma' \! - \! \gamma)  \tfrac{\log T}{2\pi}\right)  w(\gamma' \!-\! \gamma)\\
& = \  R(0) +\int_{-\infty}^{\infty} R(x)  \left\{1 - \left(\frac{\sin \pi x}{\pi x}\right)^2 \right\}\,\dx,
\end{split}
\end{align}
 where $w(x) = 4/(4+x^2)$ is a suitable weight function.
 
 \smallskip
 
Following Gallagher's \cite{G} notation, for an admissible function $R$ we define 
\begin{equation}\label{def-of-M}
M(R) := \int_{-\infty}^{\infty} R(x)  \left\{1 - \left(\frac{\sin \pi x}{\pi x}\right)^2 \right\}\,\dx,
\end{equation}
and, for $\beta >0$, we write
\begin{equation*}
\mathcal{U}(\beta):=\limsup_{T\to \infty} \frac{N(T,\beta)}{N(T)} \quad \text{ and } \quad \mathcal{L}(\beta):=\liminf_{T\to \infty}  \frac{N(T,\beta)}{N(T)}.
\end{equation*}
Let $R_{\beta}^{\pm}$ be a pair of admissible functions satisfying 
\begin{equation}\label{Intro_R_beta_pm}
R_{\beta}^-(x) \leq \chi_{[-\beta,\beta]}(x) \leq R_{\beta}^+(x)
\end{equation}
for all $x \in \R$. \new{Then, if we let 
\begin{equation*}
N^*(T) = \sum_{0<\gamma \le T} m_\gamma,
\end{equation*}
where $m_\gamma$ denotes the multiplicity of a zero of $\zeta(s)$ with ordinate $\gamma$, we observe that }
\begin{align}\label{Intro_char_beta_reduction}
\begin{split}
\frac{1}{N(T)}  \sum_{0 < \gamma, \gamma' \leq T} &  \!R_{\beta}^{+}\left( (\gamma' \! - \! \gamma)  \tfrac{\log T}{2\pi}\right)  w(\gamma' \!-\! \gamma) \\
& \geq R_{\beta}^{+}(0)\frac{N^*(T)}{N(T)} + \frac{2\, N(T, \beta)}{N(T)} + O_{\beta}\left(\frac{1}{(\log T)^2}\right)
\end{split}
\end{align}
\new{and, similarly, that
\begin{align}\label{Intro_char_beta_reduction2}
\begin{split}
\frac{1}{N(T)}  \sum_{0 < \gamma, \gamma' \leq T} &  \!R_{\beta}^{-}\left( (\gamma' \! - \! \gamma)  \tfrac{\log T}{2\pi}\right)  w(\gamma' \!-\! \gamma) \\
& \leq R_{\beta}^{-}(0)\frac{N^*(T)}{N(T)} + \frac{2\, N(T, \beta)}{N(T)} + O_{\beta}\left(\frac{1}{(\log T)^2}\right).
\end{split}
\end{align}
Observing that $N(T) \le N^*(T)$ for all $T>0$ and combining the estimates in \eqref{Mont_formula}, \eqref{def-of-M}, \eqref{Intro_R_beta_pm}, \eqref{Intro_char_beta_reduction} and \eqref{Intro_char_beta_reduction2}, we arrive at the following result.}

\begin{theorem}\label{Intro_thm1_Gallagher}
Assume RH. For any $\beta >0$ we have
\begin{equation}\label{Intro_thm1_eq1}
\frac{1}{2} M(R_{\beta}^-) \leq \mathcal{L}(\beta) \leq \mathcal{U}(\beta) \leq \frac{1}{2} M(R_{\beta}^+),
\end{equation} 
where the lower bound holds if we assume that almost all zeros of $\zeta(s)$ are simple in the sense that 
\begin{equation}\label{N star}
\lim_{T \to \infty} \frac{N^*(T)}{N(T)} = 1.
\end{equation}
\end{theorem}

This result is implicit in the work of Gallagher \cite{G}. The difficult problem here is to construct admissible majorants and minorants for $\chi_{[-\beta,\beta]}$ that optimize the values of $M(R_{\beta}^\pm)$ (and to actually compute these values). In \cite{G}, Gallagher  considered the case $\beta \in \frac12\N$, for which a classical construction of Beurling and Selberg, described in \cite{V}, produces admissible majorants and minorants $r_{\beta}^{\pm}$ that optimize the $L^1(\R)$-distance to $\chi_{[-\beta, \beta]}$ (but not necessarily the $L^1\big(\R, \big\{1 -   \big(\frac{\sin \pi x}{\pi x}\big)^2 \big\}\dx\big)$-distance). When $\beta \in \frac12\N$, the Fourier transforms $\widehat{r}_{\beta}^{\pm}$ have simple explicit representations as finite series, which allowed Gallagher to compute the values of $M(r_{\beta}^{\pm})$ and to show that 
\begin{equation}\label{Intro_BS_quotas}
\frac12 M(r_{\beta}^{\pm}) = \beta - \frac12 \pm \frac12  + \frac{1}{2 \pi^2 \beta} +  O\!\left(\frac{1}{\beta^2 }\right).
\end{equation}
In a second part of his paper \cite{G}, still in the case $\beta \in \frac12\N$, Gallagher solved the {\it two-delta problem} with respect to the pair correlation measure (i.e. to minimize $M(R)$ over the class of nonnegative admissible functions $R$ satisfying  $R(\pm \beta) \geq 1$) and was able to quantify the error between his bounds in Theorem \ref{Intro_thm1_Gallagher} and the theoretical optimal bounds achievable by this method.

\smallskip

In this paper we extend Gallagher's work \cite{G}, providing a complete solution to this problem. The three main features are:

\smallskip

\noindent (i) We find an explicit representation for the reproducing kernel associated to the pair correlation measure, which allows us to use Hilbert spaces techniques to solve the two-delta problem in the general case $\beta >0$. 

\smallskip

\noindent (ii) From the reproducing kernel, we find a suitable de Branges space of entire functions \cite{B} associated to the pair correlation measure. We solve the more general extremal problem of majorizing and minorizing characteristic functions of intervals optimizing a given de Branges metric, which provides, in particular, the optimal values of $M(R_{\beta}^{\pm})$. It turns out that asymptotics in terms of $\beta$ as in \eqref{Intro_BS_quotas} are not easily obtainable for this family, since it involves nodes of interpolation that are roots of equations with algebraic and transcendental terms. This brings us to point (iii).

\smallskip

\noindent (iii) In order to obtain (non-extremal) bounds that can be easily stated in terms of $\beta$, we compute $M(r_{\beta}^{\pm})$, for the family of Beurling-Selberg functions $r_{\beta}^{\pm}$ in the general case $\beta >0$, and prove that Gallagher's asymptotic formula in \eqref{Intro_BS_quotas} continues to hold in this case.

\smallskip

We now describe in more detail each of these three parts of the paper. We start with the third part, which is slightly simpler to state. Similar extremal problems in harmonic analysis have appeared in connection to analytic number theory, in particular to the theory of the Riemann zeta-function. For some recent results of this sort, see \cite{CC, CCM, CS, GG}.

\subsection{Explicit bounds via Beurling-Selberg majorants} Let
\begin{equation}\label{Intro_def_H_0}
H_0(z) = \left( \frac{\sin \pi z}{\pi}\right)^2 \left\{ \sum_{m=-\infty}^{\infty} \frac{\sgn(m)}{(z-m)^2} + \frac{2}{z}\right\}
\end{equation}
and 
\begin{equation}\label{Intro_def_H_1}
H_1(z) = \left( \frac{\sin \pi z}{\pi z}\right)^2.
\end{equation}
For the functions $H^\pm$ defined by $H^{\pm}(z) = H_0(z) \pm H_1(z)$, Beurling \cite{V} showed that 
\begin{equation*}
H^-(x) \leq \sgn(x) \leq H^+(x)
\end{equation*}
for all $x \in \R$, and that these are the unique extremal functions of exponential type $2\pi$ for $\sgn(x)$ (with respect to $L^1(\R)$). Moreover, we have
\begin{equation*}
\int_{-\infty}^{\infty} \big\{ H^+(x) \!-\! \sgn(x) \big\}\,\dx = \int_{-\infty}^{\infty} \big\{\sgn(x) \!-\! H^-(x)\big\}\,\dx = 1.
\end{equation*}
For $\beta >0$, Selberg \cite{V} (see also \cite{SelVol2}) considered the functions
\begin{align}\label{Intro_BS1}
\begin{split}
r_{\beta}^{+}(x):= \tfrac{1}{2} \big\{ H^{+}(x &+ \beta) + H^+(-x + \beta)\big\} \\
& \geq \tfrac{1}{2} \big\{ \sgn (x + \beta) + \sgn(-x + \beta)\big\} = \chi_{[-\beta, \beta]}(x)
\end{split}
\end{align}
and 
\begin{align}\label{Intro_BS2}
\begin{split}
r_{\beta}^{-}(x):= \tfrac{1}{2} \big\{ H^{-}(x &+ \beta) + H^-(-x + \beta)\big\} \\
& \leq \tfrac{1}{2} \big\{ \sgn (x + \beta) + \sgn(-x + \beta)\big\} = \chi_{[-\beta, \beta]}(x).
\end{split}
\end{align}
We remark that here and later, all the discontinuous functions we treat are normalized, i.e. at the discontinuity, the value of the function is the midpoint between the left-hand and right-hand limits. The functions $r_{\beta}^{\pm}$ have exponential type $2\pi$ and are bounded and integrable on $\R$. Therefore, they belong to $L^2(\R)$ and the Paley-Wiener theorem implies that they have continuous Fourier transforms supported in $[-1,1]$. Throughout the text we reserve the notation $r_{\beta}^{\pm}$ for this particular family of functions. In Section \ref{Sec_BS_majorants} we prove the following result.

\begin{theorem}\label{Intro_thm2_Gallagher_2}
Let $\beta >0$ and $r_{\beta}^{\pm}$ be the pair of admissible functions defined by \eqref{Intro_BS1} and \eqref{Intro_BS2}. Then
\begin{align}
\frac12 M(r_{\beta}^{\pm}) & = \Big( \beta \pm \frac{1}{2} \Big) - \frac{1}{2\pi^2 \beta} + \frac{\sin 2\pi\beta }{4\pi^3\beta^2} - \frac{1}{4\pi^2}  \sum_{\substack{n \in\Z }} \frac{\sgn(n^{\pm})}{(n\!-\!\beta)^2}\left( 2 + \frac{\sin 2\pi  \beta}{\pi (n \!-\! \beta)}\right) \label{Intro_value_M_r}\\
& =  \beta - \frac12 \pm \frac12  + \frac{1}{2 \pi^2 \beta} +  O\!\left(\frac{1}{\beta^2 }\right), \nonumber
\end{align} 
where $\sgn(0^{\pm}) = \pm1$.
\end{theorem}
We note that the right-hand side of \eqref{Intro_value_M_r} is a continuous function of $\beta$. In Section \ref{Sec_BS_majorants} we also include a discussion on upper and lower bounds for $N(T,\beta)$, where the parameter $\beta$ is allowed to increase as a function of $T$. 

\begin{figure} \label{figure1}
\includegraphics[scale=.44]{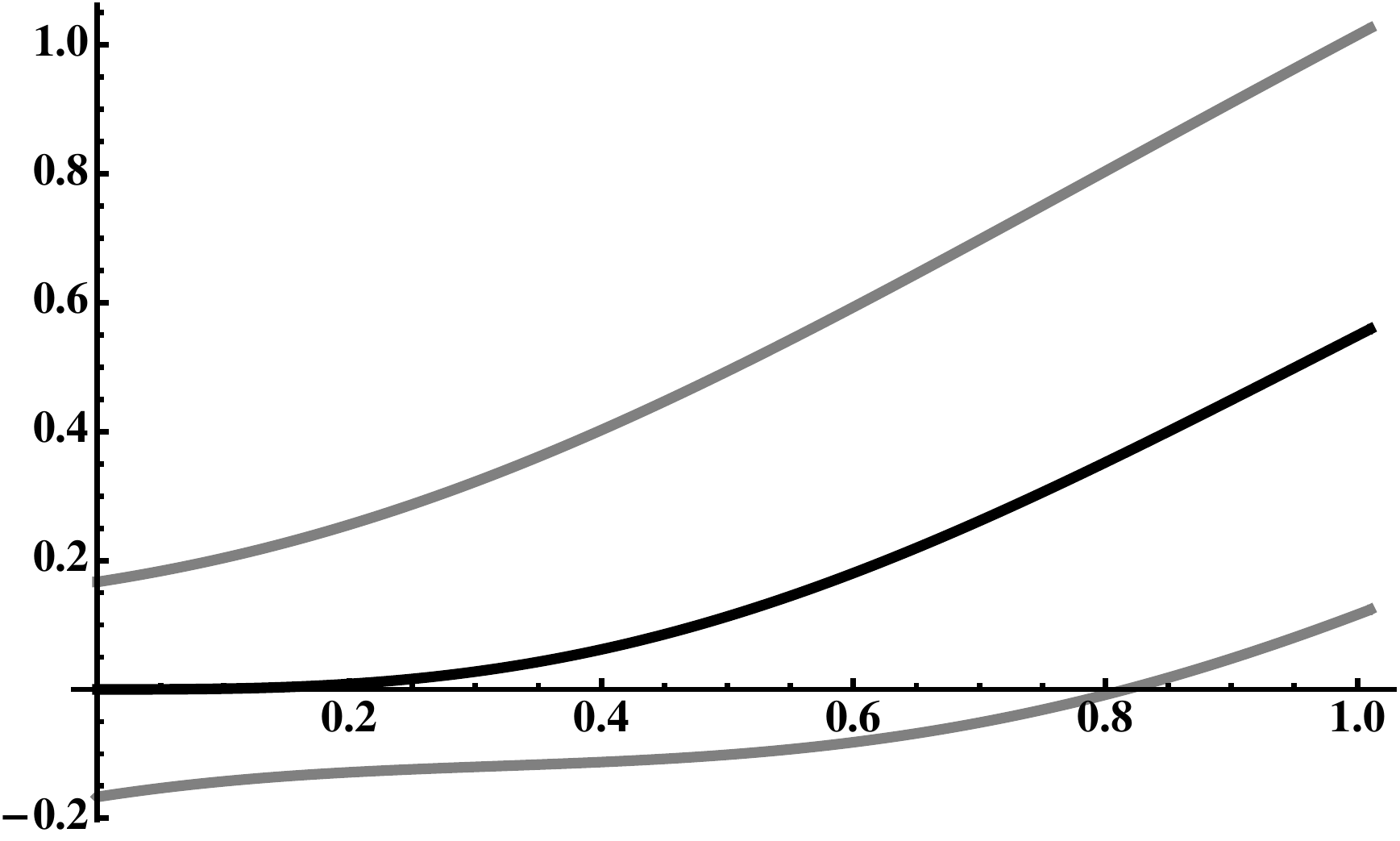} \qquad 
\includegraphics[scale=.44]{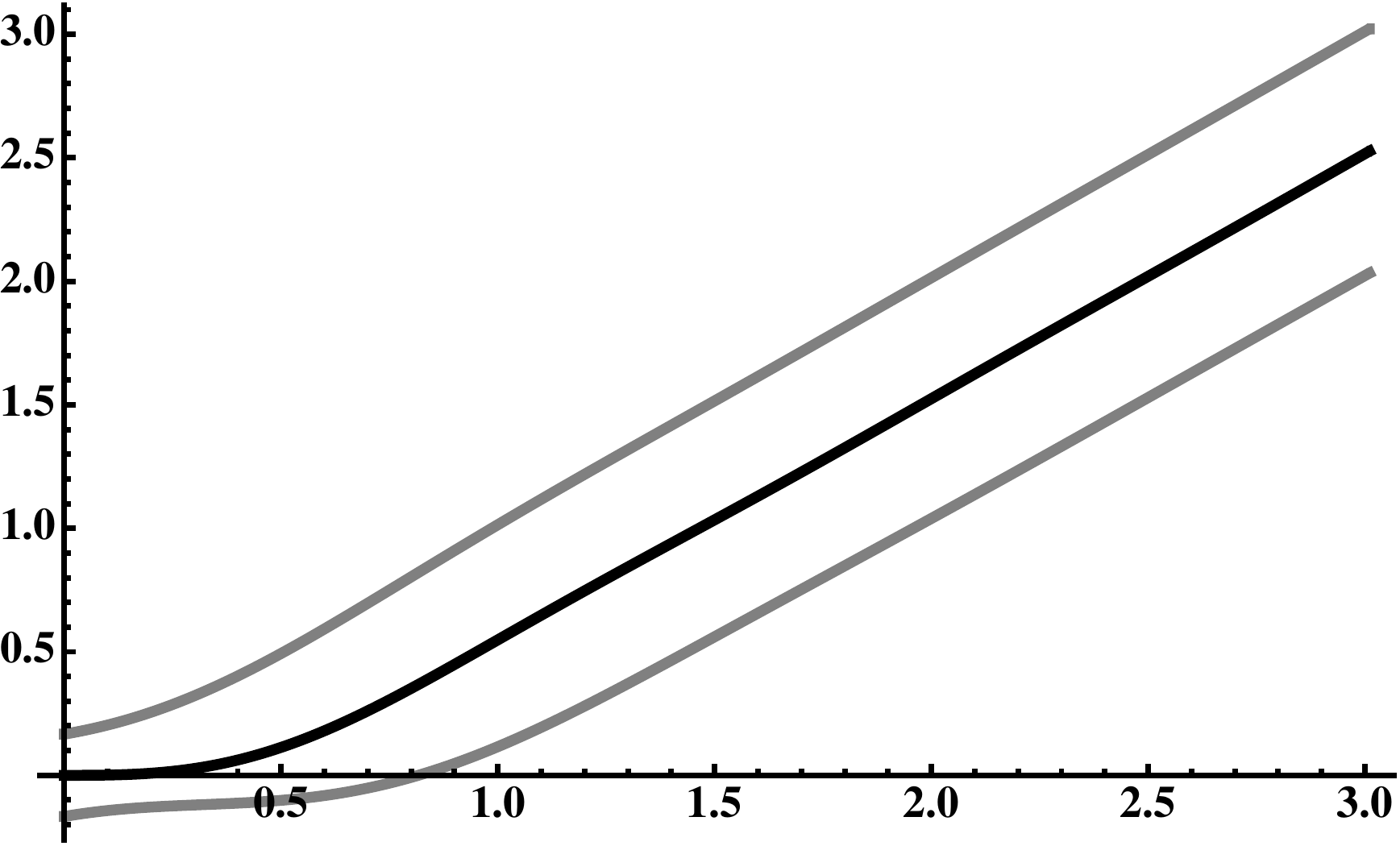} 
\caption{The above images illustrate the inequalities in Theorems \ref{Intro_thm1_Gallagher} and \ref{Intro_thm2_Gallagher_2}.  Montgomery's conjecture for $\lim_{T \to \infty} N(T,\beta)/N(T)$ is plotted in black, while the functions $\beta \mapsto \frac12 M(r_{\beta}^{\pm}) $ are plotted in gray. }
\end{figure}

\subsection{The reproducing kernel for the pair correlation measure}\label{rkhs-pcm}  

\new{The following quantity gives a lower bound for the difference of the values in Theorem \ref{Intro_thm1_Gallagher}.}  For $\beta >0$ we define
\begin{equation}\label{Intro_Delta_beta}
\varDelta(\beta) = \inf_{R \in \Omega_\beta} M(R),
\end{equation}
where the infimum is taken over the subclass $\Omega_\beta$ of nonnegative admissible functions $R$ such that $R(\pm \beta) \geq1$. If $R_{\beta}^{\pm}$ is a pair of admissible functions satisfying \eqref{Intro_R_beta_pm} then $R:= (R_{\beta}^+ - R_{\beta}^-) \in \Omega_\beta$ and 
\begin{equation*}
M(R_{\beta}^+) - M(R_{\beta}^-)  = M(R) \geq \varDelta(\beta).
\end{equation*}
Hence the gap between an upper bound for $\mathcal{U}(\beta)$ and a lower bound for $\mathcal{L}(\beta)$ in Theorem \ref{Intro_thm1_Gallagher} cannot be smaller than $\frac{1}{2}\varDelta(\beta)$.

\smallskip

In the case $\beta \in \frac12\N$, Gallagher \cite[Section 2]{G} used a variational argument to solve this two-delta problem and compute $\varDelta(\beta)$. This argument was previously used by Montgomery and Taylor \cite{M2} to solve the simpler one-delta problem in connection to bounds for the proportion of simple zeros of $\zeta(s)$. Gallagher's variational approach for the two-delta problem relies heavily on the fact that $\beta \in \frac12\N$ to establish orthogonality relations in some passages, thus making its extension to the general case $\beta >0$ a nontrivial task. Here we revisit this problem and solve it in the general case using a different technique, namely the theory of reproducing kernel Hilbert spaces.  Proofs of the theorems in this section are given in Section \ref{HS_approach}.

Let us write 
$$\dmu(x) =  \left\{ 1 - \left(\frac{\sin \pi x}{\pi x}\right)^2\right\} \,\dx.$$
We denote by $\mc{B}_2(\pi,\mu)$ the class of entire functions $f$ of exponential type at most $\pi$ for which
$$\int_{-\infty}^\infty |f(x)|^2 \, \dmu(x) <\infty,$$
and we write $\mc{B}_2(\pi)$ if $\dmu$ is replaced by the Lebesgue measure (i.e. $\mc{B}_2(\pi)$ is the classical Paley-Wiener space). Using the uncertainty principle for the Fourier transform, we show that $\mu$ and the Lebesgue measure define equivalent norms on the class of functions of exponential type at most $\pi$ for which either and hence both norms are finite. This implies, in particular, that $\mc{H} = \mc{B}_2(\pi,\mu)$ is a Hilbert space with norm given by 
$$\|f\|^2_{\mc{H}} = \int_{-\infty}^\infty |f(x)|^2 \, \dmu(x).$$
For each $w \in \C$, the functional $f \mapsto f(w)$ is therefore continuous on $\mc{H}$ (since this holds for the Paley-Wiener space $\mc{B}_2(\pi)$). Hence, there exists a function $K(w,\cdot) \in \mc{H}$ such that 
$$f(w) = \langle f, K(w,\cdot) \rangle_{\mc{H}}  =  \int_{-\infty}^\infty f(x)\, \overline{K(w,x)}\, \dmu(x)$$
for all $f \in \mc{H}$. This is the so-called {\it reproducing kernel} for the Hilbert space $\mc{H}$, and our first goal is to find an explicit representation for this kernel. For $w\in\C$ (initially with $w \neq \pm 1/\pi\sqrt{2}$) define constants $c(w)$ and $d(w)$ by
\begin{align}\label{Intro_HS_eq3}
\begin{split}
c(w) &= \frac{\cos(\pi w) - \pi w \sin(\pi w) }{(1-2\pi^2 w^2 ) \big(\cos\big(2^{-\frac12}\big) - 2^{-\frac12} \sin\big(2^{-\frac12}\big)\big)},\\
d(w) &= \frac{ 2\pi  w \cos(\pi w) }{(1-2\pi^2 w^2 )\, 2^{\frac12} \cos\big(2^{-\frac12}\big) },
\end{split}
\end{align}
and functions $f(w,\cdot) ,g,h\in \mc{H}$ by
\begin{align*}
f(w,z) &=  \frac{2\pi^2  w^2}{(2\pi^2 w^2 - 1)} \frac{\sin\pi(z-w)}{\pi(z-w)}, \\
g(z) &= \frac{2^{\frac12}\sin\big(2^{-\frac12}\big) \cos(\pi z) - 2\pi z \cos\big(2^{-\frac12}\big) \sin(\pi z)}{1- 2\pi^2 z^2},  \\
h(z) &= \frac{2\pi z \sin\big(2^{-\frac12}\big)\cos(\pi z)  - 2^{\frac12} \cos\big(2^{-\frac12}\big) \sin(\pi z)}{1-2\pi^2 z^2}.
\end{align*}

\begin{theorem}\label{Intro_HS_Thm1_RP}
 For each $w\in \C$ we have 
\begin{align}\label{Intro_rk-rep}
K(w,z) &= f(\ov{w},z) + c(\ov{w}) g(z) + d(\ov{w}) h(z).
\end{align}
\new{At the points $w = \pm 1/\pi\sqrt{2}$, this formula should be interpreted in terms of the appropriate limit.}
\end{theorem}


\smallskip

We exploit the Hilbert space structure and the explicit formula for the reproducing kernel  to give a complete solution to the two-delta problem with respect to the pair correlation measure. 
\begin{theorem}\label{Intro_HS_Thm5}
Let $\beta >0$,  let $\varDelta(\beta)$ be defined by \eqref{Intro_Delta_beta}, \new{and let $K$ be given by \eqref{Intro_rk-rep}.} Then
\begin{align}\label{Intro_HS_Thm2_eq1}
\begin{split}
\varDelta(\beta) = \frac{2}{ K(\beta, \beta) + |K(\beta, -\beta)|}= 2 \left\{ 1 - \left|\frac{\sin2\pi \beta}{2\pi\beta}\right| \right\} + O\!\left(\frac{1}{\beta^2 }\right).
\end{split}
\end{align}
The extremal functions {\rm(}i.e. functions that realize the infimum in \eqref{Intro_Delta_beta}{\rm )} are given by the following formulae.
\begin{enumerate}
\item[(i)]If $K(\beta, -\beta) = 0$, then
\begin{equation*}
R(z) = \frac{1}{K(\beta, \beta)^2} \big( c_1 K(\beta, z) + c_2 K(-\beta, z) \big)  \big( \ov{c_1} \,K(\beta, z) + \ov{c_2} \,K(-\beta, z) \big),
\end{equation*}
where $c_1, c_2 \in \C$ with $|c_1| = |c_2| =1$. 
\smallskip
\item[(ii)] If $K(\beta, -\beta) \neq 0$, then
\begin{equation*}
R(z) = \frac{\left(\frac{K(\beta, -\beta)}{|K(\beta, -\beta)|} K(\beta, z) + K(-\beta, z)\right)^2}{\big(K(\beta, \beta) + |K(\beta, -\beta)|\big)^2}.
\end{equation*}
\end{enumerate}
\end{theorem}

\smallskip

In particular, the bounds given in Theorem \ref{Intro_thm2_Gallagher_2} are optimal up to order $O(\beta^{-2})$ when $\beta \in \frac12 \N$. The appearance of the term $|\frac{\sin2\pi \beta}{2\pi\beta}|$ on the right-hand side of \eqref{Intro_HS_Thm2_eq1} is not a coincidence, for this term already appears naturally in the work of Littmann \cite{Lit} on the Beurling-Selberg extremal problem for $\chi_{[-\beta, \beta]}(x)$. Using the same circle of ideas, one could explicitly compute the reproducing kernels associated to other measures that arise naturally in the study of families of $L$-functions, see \cite{ILS, KS}.

\subsection{An extremal problem in de Branges spaces} 

\subsubsection{De Branges spaces} Let us briefly review the basic facts and terminology of de Branges' theory of Hilbert spaces of entire functions \cite[Chapters 1 and 2]{B}.  A function $F$ analytic in the open upper half-plane $\C^{+} = \{z \in \C;\ \im(z) >0\}$ has {\it bounded type} if it can be written as the quotient of two functions that are analytic and bounded in $\C^{+}$. If $F$ has bounded type in $\C^{+}$, from its Nevanlinna factorization \cite[Theorems 9 and 10]{B} we have
\begin{equation*}
v(F) = \limsup_{y \to \infty} \, y^{-1}\log|F(iy)| <\infty.
\end{equation*}
The number $v(F)$ is called the {\it mean type} of $F$. If $F:\C \to \C$ is entire, we denote by $\tau(F)$ its exponential type, i.e. 
\begin{equation*}
\tau(F) = \limsup_{|z|\to \infty} |z|^{-1}\log|F(z)|,
\end{equation*}
and we define $F^*:\C \to \C$ by $F^*(z) = \overline{F(\overline{z})}$. We say that $F$ is {\it real entire} if $F$ restricted to $\R$ is real-valued.

\smallskip

Let $E:\C \to \C$ be a {\it  Hermite-Biehler} function, i.e. an entire function satisfying the basic inequality
\begin{equation}\label{Intro_HB_cond}
|E(\overline{z})| < |E(z)|
\end{equation}
for all $z \in \C^+$. The {\it de Branges space} $\H(E)$ is the space of entire functions $F:\C \to \C$ such that
\begin{equation}\label{Intro_dB_inequality1}
\|F\|_E^2 := \int_{-\infty}^\infty |F(x)|^{2} \, |E(x)|^{-2} \, \dx <\infty\,,
\end{equation}
and such that $F/E$ and $F^*/E$ have bounded type and nonpositive mean type in $\C^{+}$. The remarkable property  about $\H(E)$ is that it is a reproducing kernel Hilbert space with inner product 
\begin{equation*}
\langle F, G \rangle_{E} =  \int_{-\infty}^\infty F(x) \, \ov{G(x)} \, |E(x)|^{-2} \, \dx.
\end{equation*}
The reproducing kernel (that we continue denoting by $K(w,\cdot)$) is given by (see \cite[Theorem 19]{B})
\begin{equation}\label{Intro_dB_rk}
2\pi i (\ov{w}-z)K(w,z) = E(z)E^*(\ov{w}) - E^*(z)E(\ov{w}). 
\end{equation}
Associated to $E$, we consider a pair of real entire functions $A$ and $B$ such that  $E(z) = A(z) -iB(z)$. These functions are given by 
\begin{equation*}
A(z) := \frac12 \big\{E(z) + E^*(z)\big\} \ \ \ {\rm and}  \ \ \ B(z) := \frac{i}{2}\big\{E(z) - E^*(z)\big\},
\end{equation*}
and the reproducing kernel has the alternative representation 
\begin{equation*}
\pi (z - \ov{w})K(w,z) = B(z)A(\ov{w}) - A(z)B(\ov{w}).
\end{equation*}
When $z = \ov{w}$ we have
\begin{equation}\label{Intro_Def_K}
\pi K(\ov{z}, z) = B'(z)A(z) - A'(z)B(z).
\end{equation}
For each $w \in \C$, the reproducing kernel property implies that 
\begin{align*}
0 \leq \|K(w, \cdot)\|_E^2 = \langle K(w, \cdot), K(w, \cdot) \rangle_E = K(w,w),
\end{align*}
and it is not hard to show (see \cite[Lemma 11]{HV}) that $K(w,w)=0$ if and only if $w \in\R$ and $E(w) = 0$ (in this case we have $F(w) =0$ for all $F \in \H(E)$). 

\smallskip

For our purposes we consider the class of Hermite-Biehler functions $E$ satisfying the following properties:
\begin{enumerate}
\item[(P1)] $E$ has bounded type in $\C^{+}$;
\item[(P2)] $E$ has no real zeros; 
\item[(P3)] $z \mapsto E(iz)$ is a real entire function;
\item[(P4)] $A, B \notin \H(E)$. 
\end{enumerate} 
By a result of M. G. Krein (see \cite{K} or \cite[Lemmas 9 and 12]{HV}) we see that if $E$ satisfies (P1), then $E$ has exponential type and $\tau(E) = v(E)$. Moreover, the space $\H(E)$ consists of the entire functions $F$ of exponential type $\tau(F) \leq \tau(E)$ that satisfy \eqref{Intro_dB_inequality1}.

\subsubsection{\new{De Branges space for the pair correlation measure}} We show that the Hilbert space $\mc{H}$ defined in Section \ref{rkhs-pcm} can be identified  with a suitable de Branges space $\mc{H}(E)$, where $E$ is a Hermite-Biehler function satisfying properties (P1) - (P4). Define
\begin{equation*}
L(w,z) = 2\pi i (\overline{w} - z) K(w,z)\,,
\end{equation*}
where $K$ is given \new{by \eqref{Intro_rk-rep}}. It follows then that the entire function 
\begin{equation}\label{Intro_Def_E_special}
E(z) = \frac{L(i,z)}{L(i,i)^{\frac12}}
\end{equation}
is a Hermite-Biehler function such that 
\begin{equation}\label{Intro_rep_kernel}
L(w,z) = E(z) E^*(\ov{w}) -  E^*(z)E(\ov{w}).
\end{equation}
For the convenience of the reader we include short proofs of these facts in Appendix A.  This implies \cite[Theorem 23]{B} that the Hilbert space $\mc{H}$ is isometrically equal to the de Branges space $\mc{H}(E)$. In particular, the key identity
\begin{equation}\label{Intro_key_id}
\int_{-\infty}^{\infty} |f(x)|^2 \, |E(x)|^{-2}\,\dx = \int_{-\infty}^\infty |f(x)|^2 \, \dmu(x) 
\end{equation}
holds for any $f \in \mc{H}$.

\smallskip

 We now verify (P1) - (P4). It is clear that $E(z)$ has exponential type $\pi$ and is bounded on $\R$. Therefore, by the converse of Krein's theorem (see \cite{K} or \cite[Lemma 9]{HV}), we have that $E$ has bounded type in $\C^+$,  which shows (P1). If $E$ had a real zero $w$, we would have $F(w) = 0$ for all $F \in \H(E) = \H$. However, we have seen that $\H$ is equal (as a set) to the Paley-Wiener space, which is a contradiction. This proves (P2). 

\smallskip 

A direct computation using \eqref{Intro_Def_E_special} and Theorem \ref{Intro_HS_Thm1_RP} shows that $E(ix)$ is real when $x$ is real, which shows (P3). For real $x$ we have $A(x) = \re (E(x))$ and $B(x) = - \im (E(x))$. Since $c(-i), id(-i), g(x)$ and $h(x)$ are all real, a direct computation gives us
\begin{align*}
\begin{split}
A(x) &= \frac{\re(L(i,x))}{L(i,i)^{\frac12}}\\
& = \frac{1}{L(i,i)^{\frac12}} \frac{4 \pi^2}{(2\pi^2 +1)} \cos\pi x \left\{\sinh \pi  + \frac{\tan\big(2^{-\frac12}\big)\,\cosh \pi}{\pi \sqrt{2}}\right\} + O(x^{-1})
\end{split}
\end{align*}
and
\begin{align*}
\begin{split}
B(x) &= - \frac{\im(L(i,x))}{L(i,i)^{\frac12}}\\
& = \frac{1}{L(i,i)^{\frac12}} \frac{4 \pi^2}{(2\pi^2 +1)} \sin\pi x \left\{\cosh \pi  + \frac{(\cosh \pi + \pi \sinh \pi)\cos\big(2^{-\frac12}\big) }{2 \pi^2 \big(\!\cos\big(2^{-\frac12}\big) - 2^{-\frac12} \sin\big(2^{-\frac12}\big)\big)}\right\} + O(x^{-1}),
\end{split}
\end{align*}
for large $x$. This shows that $A,B \notin L^2(\R)$ and thus, by \eqref{Intro_key_id} and Lemma \ref{Lem_equiv_norms} below, $A,B \notin \H(E)$. This proves (P4).

\subsubsection{The extremal problem} \new{We now return to the case of an arbitrary Hermite-Biehler function $E$ satisfying properties (P1) - (P4) above.} From now on we assume, without loss of generality, that $E(0) >0$ (note that this holds for the particular $E$ defined by \eqref{Intro_Def_E_special}). \new{Generalizing \eqref{def-of-M},} let us write
\begin{equation*}
M_{E}( R )= \int_{-\infty}^{\infty} R(x)\,|E(x)|^{-2}\,\dx.
\end{equation*}
For $\beta >0$ we define
\begin{equation}\label{Intro_Def_Lambda_+}
\varLambda_E^+(\beta) = \inf {M_E(R_{\beta}^+)},
\end{equation}
and
\begin{equation}\label{Intro_Def_Lambda_-}
\varLambda_E^-(\beta) = \sup {M_E(R_{\beta}^-)},
\end{equation}
where the infimum and the supremum are taken over the entire functions $R_{\beta}^{\pm}$ of exponential type at most $2\tau(E)$ such that
\begin{equation}\label{char-ineq-E}
R_{\beta}^-(x) \leq \chi_{[-\beta,\beta]}(x) \leq R_{\beta}^+(x)
\end{equation}
for all $x \in \R$.

\smallskip

In its simplest version, for the Paley-Wiener space (which corresponds to $E(z) = e^{-i\pi z}$), this is a classical problem in harmonic analysis with numerous applications to inequalities in number theory and signal processing. Its sharp solution was discovered by Beurling and Selberg \cite{V} when $\beta \in \hh\N$, by Donoho and Logan \cite{DL} when $\beta < \frac12$, and recently by Littmann \cite{Lit} for the remaining cases\footnote{B. F. Logan announced the solution for the general case in the abstract ``Bandlimited functions bounded below over an interval", Notices Amer. Math. Soc., 24 (1977), pp. A331. His proof, however, has never been published.}. Here we provide a complete solution to this optimization problem with respect to a general de Branges metric $L^1(\R, |E(x)|^{-2}\,\dx)$. As in the Paley-Wiener case, there are three distinct qualitative regimes for the solution, and these depend on the roots of $A$ and $B$ (observe that if $E(z) = e^{-i\pi z}$, then $A(z) = \cos \pi z$ and $B(z) = \sin \pi z$, which have roots exactly at $\beta \in \hh \N$). Similar extremal problems in de Branges and Euclidean spaces were considered in \cite{CGon, CL2, HV, Ke, LS}.

\smallskip

Property (P3) implies that $A$ is even and $B$ is odd, and by the Hermite-Biehler condition, $A$ and $B$ have only real zeros. Morever, these zeros are all simple. To see this, note that by \eqref{Intro_Def_K} we see that any double zero $w$ of either $A$ or $B$ implies that $K(w,w)=0$ which would, in turn, imply that $E(w) =0$ in contradiction to (P2). It also follows from well-known properties of Hermite-Biehler functions (see for instance the discussion related to the phase function in \cite[Problem 48]{B} or \cite[Section 3]{HV}) that the zeros of $A$ and $B$ interlace. In our case we have $B(0) = 0$ and $A(0) >0$. If we label the nonnegative zeros of $B$ in order as $0 = b_0 < b_1 < b_2 < \ldots $ and the positive zeros of $A$ as $a_1 < a_2 < \ldots $, then we have
\begin{equation*}
0 = b_0 < a_1 < b_1 < a_2 < b_2 < \ldots 
\end{equation*}

For each $\beta > 0$ that is not a root of $A$ or $B$, we define an auxiliary Hermite-Biehler function $E_{\beta}(z)$. The corresponding companion functions $A_{\beta}(z)$ and $B_{\beta}(z)$ and the reproducing kernel $K_{\beta}(w,z)$ play an important role in the solution of our extremal problem. We divide this construction in two cases, depending on the sign of $A(\beta)B(\beta)$. Since $A(0) >0$ and $B(0)=0$, from \eqref{Intro_Def_K} we find that $B'(0) > 0$. Then,

\smallskip

\noindent(i) if $b_k < \beta < a_{k+1}$, we set  $\gamma_\beta:= \beta B(\beta)/A(\beta) >0;$

\smallskip

\noindent(ii) if $a_k < \beta < b_k$, we set $\gamma_\beta:= -\beta A(\beta)/B(\beta) >0.$

\smallskip

\noindent In either case we now define $E_\beta$ by
\begin{equation}\label{Intro_Def_E_beta_2}
E_{\beta}(z) = E(z)(\gamma_{\beta} - iz).
\end{equation}
\begin{theorem}\label{Intro_thm5_super}
Let $E$ be a Hermite-Biehler function satisfying properties {\rm (P1) - (P4)}. Let $\beta >0$ and $\varLambda_E^{\pm}(\beta)$ be defined by \eqref{Intro_Def_Lambda_+} and \eqref{Intro_Def_Lambda_-}. 
\begin{enumerate}
\item[(i)] If $\beta \in \{a_i\}$, then
\begin{equation*}
\varLambda_E^+(\beta) = \sum_{\stackrel{A(\xi) = 0}{|\xi| \leq \beta}} \frac{1}{K(\xi, \xi)} \ \ \ {\rm and} \ \ \ \varLambda_E^-(\beta) = \sum_{\stackrel{A(\xi) = 0}{|\xi| < \beta}} \frac{1}{K(\xi, \xi)}.
\end{equation*}
\item[(ii)] If $\beta \in \{b_i\}$, then
\begin{equation*}
\varLambda_E^+(\beta) =  \sum_{\stackrel{B(\xi) = 0}{|\xi| \leq \beta}} \frac{1}{K(\xi, \xi)} \ \ \ {\rm and} \ \ \ \varLambda_E^-(\beta) =  \sum_{\stackrel{B(\xi) = 0}{|\xi| < \beta}} \frac{1}{K(\xi, \xi)}.
\end{equation*}
\item[(iii)] If $b_k < \beta < a_{k+1}$, then
\begin{equation*}
\varLambda_E^+(\beta) =  \sum_{\stackrel{A_{\beta}(\xi) = 0}{|\xi| \leq \beta}} \frac{\xi^2 + \gamma_{\beta}^2}{K_{\beta}(\xi, \xi)}\ \ \ {\rm and} \ \ \ \varLambda_E^-(\beta) =  \sum_{\stackrel{A_{\beta}(\xi) = 0}{|\xi| < \beta}} \frac{\xi^2 + \gamma_{\beta}^2}{K_{\beta}(\xi, \xi)}.
\end{equation*}
\item[(iv)] If $a_k < \beta < b_{k}$, then
\begin{equation*}
\varLambda_E^+(\beta) =  \sum_{\stackrel{B_{\beta}(\xi) = 0}{|\xi| \leq \beta}} \frac{\xi^2 + \gamma_{\beta}^2}{K_{\beta}(\xi, \xi)}\ \ \ {\rm and} \ \ \ \varLambda_E^-(\beta) =  \sum_{\stackrel{B_{\beta}(\xi) = 0}{|\xi| < \beta}} \frac{\xi^2 + \gamma_{\beta}^2}{K_{\beta}(\xi, \xi)}.
\end{equation*}
\end{enumerate}
In each of the cases above, there exists a pair of extremal functions \new{ $R_{\beta,E}^\pm$, i.e. functions for which \eqref{char-ineq-E} holds and  the identities $M_E(R_{\beta,E}^\pm) = \Lambda_E^\pm(\beta)$ are valid. In particular, the values $M_E(R_{\beta,E}^\pm)$ are finite.} These extremal functions interpolate the characteristic function $\chi_{[-\beta,\beta]}$ at points $\xi$ given by {\rm (i)} $A(\xi) = 0$; {\rm (ii)} $B(\xi) = 0$; {\rm (iii)} $A_{\beta}(\xi) = 0$; {\rm (iv)} $B_{\beta}(\xi) = 0$, respectively. In the generic cases  {\rm (iii)}  and  {\rm (iv)}  such a pair of extremal functions is unique.
\end{theorem}

\noindent{\sc Remark.} In the above theorem, interpolating $\chi_{[-\beta,\beta]}$ at the endpoints $\xi = \pm \beta$ means taking the value $1$ for the majorant and the value $0$ for the minorant.

\smallskip

We observe that Theorem \ref{Intro_thm5_super} provides a complete solution to our original extremal problem related to the pair correlation measure. In fact, recall that $E$ defined by \eqref{Intro_Def_E_special} has exponential type $\pi$. Let $R_{\beta}^{\pm}$ be a pair of functions of exponential type at most $2\pi$ that verifies \eqref{Intro_R_beta_pm}. Since $R_{\beta}^+$ is nonnegative on $\R$, a classical result of Krein \cite[p. 154]{A} (alternatively, see \cite[Lemma 14]{CL2}) gives us the representation $R_{\beta}^+(z) = U(z)U^*(z)$, where $U$ is entire of exponential type at most $\pi$. By the identity \eqref{Intro_key_id} we have
\begin{equation*}
M(R_{\beta}^+) = \int_{-\infty}^{\infty} |U(x)|^2 \,\dmu(x) = \int_{-\infty}^{\infty} |U(x)|^2 \, |E(x)|^{-2}\, \dx = M_E(R_{\beta}^+)
\end{equation*}
provided either, and hence both, of the values $M(R_{\beta}^+)$ or $M_E(R_{\beta}^+)$ is finite. To prove the analogous statement for $R_{\beta}^-$, we  \new{write  $R_\beta^-$ as a difference of nonnegative functions (on $\R$), }
\begin{equation*}
\new{R_{\beta}^-(z) =R_{\beta,E}^+(z) - \big(R_{\beta,E}^+(z) - R_{\beta}^-(z)\big) },
\end{equation*}
and conclude that
\begin{equation*}
M(R_{\beta}^-) = M(R_{\beta,E}^+) - M(R_{\beta,E}^+ - R_{\beta}^-) = M_E(R_{\beta,E}^+) - M_E(R_{\beta,E}^+ - R_{\beta}^-) = M_E(R_{\beta}^-)
\end{equation*}
provided either, and hence both, of the values $M(R_{\beta}^-)$ or $M_E(R_{\beta}^-)$ is finite.

\subsubsection{Connection to the two-delta problem} We may consider the two-delta problem in the general de Branges setting, i.e. for a Hermite-Biehler function $E$ satisfying properties {\rm (P1) - (P4)} we define
\begin{equation}\label{two-Delta_E}
\varDelta_E(\beta) = \inf_{R \in \Omega_{\beta,E}} M_E(R)
\end{equation}
where the infimum is taken over the subclass $\Omega_{\beta,E}$ of nonnegative functions $R$ of exponential type at most $2\tau(E)$ such that $R(\pm \beta) \geq1$. Since $\H(E)$ is a reproducing kernel Hilbert space, the solution for this problem is given by Theorem \ref{Intro_HS_Thm5} (the proof is identical, with $K$ being the reproducing kernel of the space $\H(E)$). If $R_{\beta,E}^\pm$ is a pair of extremal functions given by Theorem \ref{Intro_thm5_super}, we show in Section \ref{Sec_de_Branges_spaces} that their difference $R:= R_{\beta,E}^+ - R_{\beta,E}^-$ is an extremal function for the two-delta problem \eqref{two-Delta_E}, and in particular we obtain
\begin{equation*}
\varDelta_E(\beta)  = \varLambda_E^+(\beta) - \varLambda_E^-(\beta). 
\end{equation*} 
From Theorem \ref{Intro_thm1_Gallagher} and Theorem \ref{Intro_HS_Thm5} we arrive at the following result.

\begin{corollary}\label{Intro_Cor6}
Assume RH and \eqref{N star}, and let $K(w,z)$ be defined by \eqref{Intro_rk-rep}. Then 
\begin{equation} \label{U minus L}
\big\{\U(\beta) - \mathcal{L}(\beta)\big\} \leq  \frac{1}{ K(\beta, \beta) + |K(\beta, -\beta)|}=  1 - \left|\frac{\sin2\pi \beta}{2\pi\beta}\right| + O\!\left(\frac{1}{\beta^2 }\right).
\end{equation}
\end{corollary}

\subsection{Related results} 
Our lower bounds for $N(T,\beta)$ are only nontrivial if the left-hand side of the inequality in \eqref{Intro_thm1_eq1} is positive. It is natural to ask for bounds on the smallest value of $\beta$ for which $N(T,\beta)$ is positive. For instance, in the context of Theorem \ref{Intro_thm2_Gallagher_2}, a straightforward numerical calculation implies that $ \tfrac12M(r_{\beta}^-) >0$ if $\beta \ge 0.8163$ and hence, assuming RH and \eqref{N star}, we see that $N(T,0.8163) \gg N(T)$; this is illustrated in Figure 1. In Section \ref{sg}, using Montgomery's formula in a different manner, we improve this estimate.

\begin{theorem}\label{small gaps}
Assume RH and \eqref{N star}. Then $N(T,0.606894) \gg N(T).$
\end{theorem}

As stated, this result appears to be the best known result on small gaps coming from Montgomery's formula. Theorem \ref{small gaps} gives a modest improvement of the previous results of Montgomery \cite{M1} and Goldston, Gonek, \"{O}zl\"{u}k and Snyder \cite{GGOS} who, under the same assumptions, had shown that $N(T,0.6695...) \gg N(T)$ and $N(T, 0.6072...) \gg N(T)$, respectively.\footnote{The result in \cite{M1} is stated with $0.68$ in place of $0.6695...$\,. As is pointed out in \cite{GGOS}, it is not difficult to modify Montgomery's argument to derive this sharper estimate. Moreover, it is shown in \cite{GGOS} that a result stronger than Theorem \ref{small gaps} holds assuming \eqref{N star} and the generalized Riemann hypothesis for Dirichlet $L$-functions.} Our proof differs somewhat from the proofs of these previous results since we actually use Montgomery's formula twice, choosing two different test functions.

\smallskip

\begin{figure} \label{figure2}
\includegraphics[scale=.44]{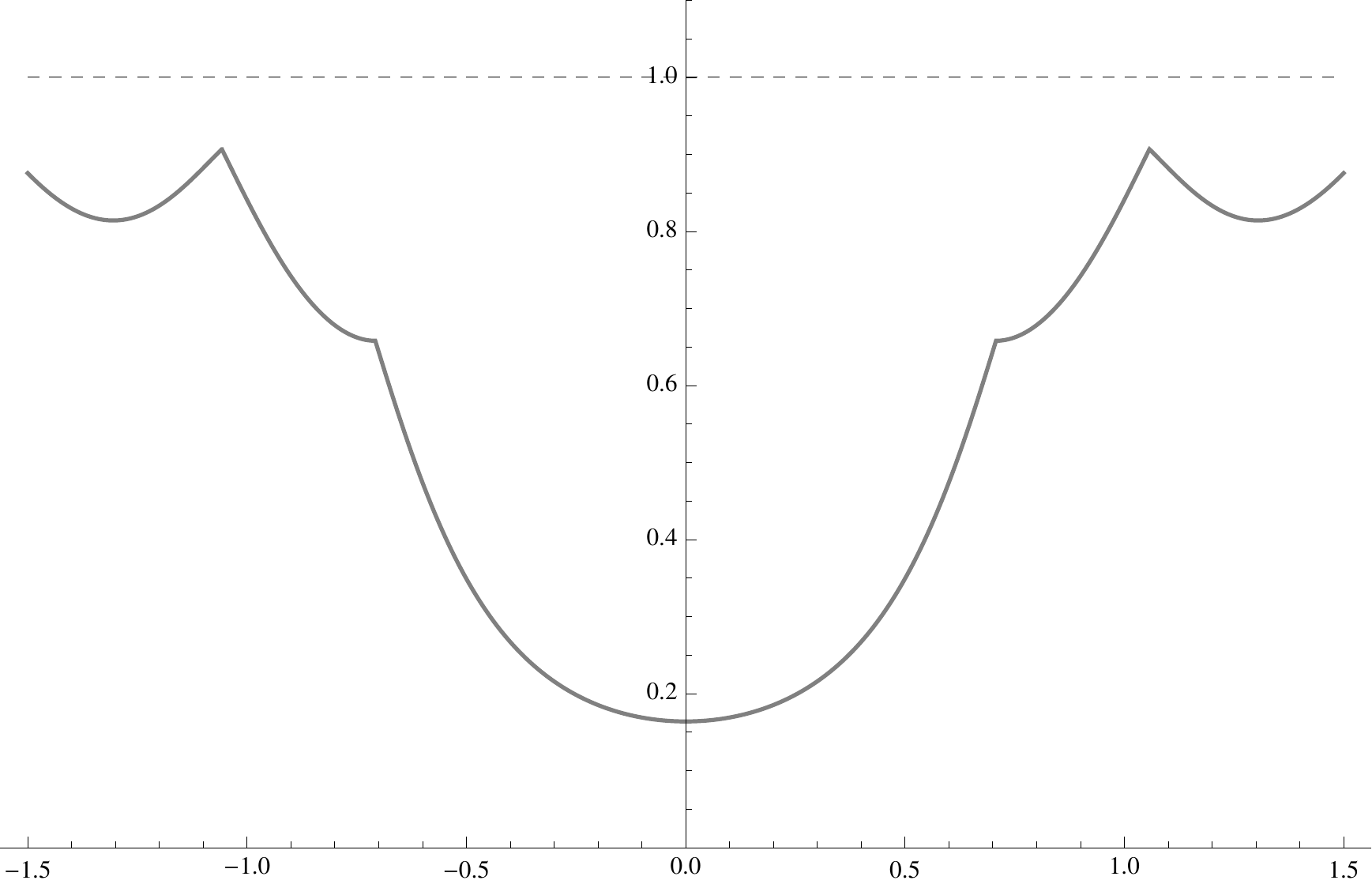} \qquad 
\includegraphics[scale=.44]{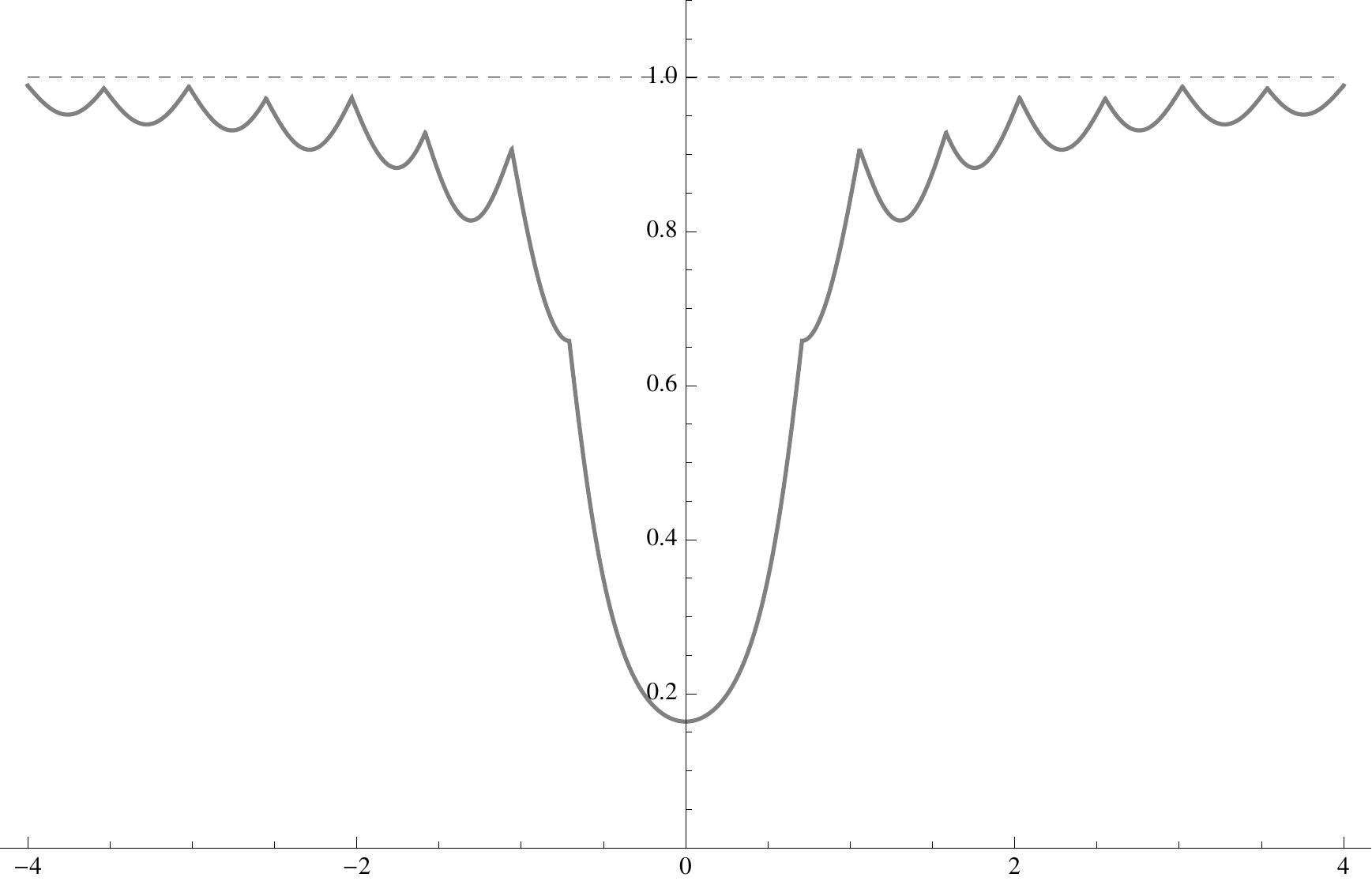} 
\caption{The above images illustrate the upper bound for $\U(\beta) - \mathcal{L}(\beta)$ given in Corollary \ref{Intro_Cor6}. }
\end{figure}

Theorem \ref{small gaps} implies that infinitely often the gap between the imaginary parts of consecutive nontrivial zeros of $\zeta(s)$ is less than the average spacing. Define the quantity
\[
\mu = \liminf_{n \to \infty} \frac{(\gamma_{n+1}\!-\!\gamma_n) \log \gamma_n}{2\pi}.
\]
Since the average size of $\gamma_{n+1}-\gamma_n$ is $2\pi/\log \gamma_n$, we see that trivially $\mu \le 1$. Assuming RH, Theorem \ref{small gaps} implies that $\mu \le 0.606894$. To see why, note that if \eqref{N star} holds then the claimed inequality for $\mu$ follows from Theorem \ref{small gaps} since $\mu \le \beta$ if $N(T,\beta)\gg N(T)$. On the other hand, if \eqref{N star} does not hold, then there are infinitely many multiple zeros of $\zeta(s)$ implying that $\mu=0$. Hence, in either case, we have $\mu \le 0.606894$.

\smallskip

Due to the connection to the class number problem for imaginary quadratic fields \cite{CI,MW}, it is an interesting open problem to prove that $\mu < \frac{1}{2}$. By a different method, also assuming RH, Feng and Wu \cite{FW} have proved that $\mu \le 0.5154$. This improves previous estimates by a number of other authors \cite{BMN,CCG,MO}. It does not appear, however, that any of these results can be applied to prove nontrivial estimates for the function $N(T,\beta)$.

\smallskip

In Section \ref{Sec_Q_analogue}, we prove a result which is an analogue of Theorems \ref{Intro_thm1_Gallagher} and \ref{Intro_thm2_Gallagher_2} for the zeros of primitive Dirichlet $L$-functions in $q$-aspect. This requires the version of Montgomery's formula given in \cite{CLLR}, which was proved using a modification of the asymptotic large sieve of Conrey, Iwaniec and Soundararajan \cite{CIS1}. In this case, the results in \cite{CLLR} allow to use Beurling-Selberg  majorants and minorants of $\chi_{[-\beta, \beta]}(x)$ with Fourier transforms supported in $(-2,2)$. This leads to stronger results which are stated in Theorem \ref{q theorem}.

\section{Bounds via Beurling-Selberg majorants}\label{Sec_BS_majorants}

In this section we prove Theorem \ref{Intro_thm2_Gallagher_2}. Exploiting the fact that we have explicit expressions for the Beurling-Selberg functions $r_{\beta}^{\pm}$ and their Fourier transforms, we also prove a version of Theorem \ref{Intro_thm1_Gallagher} that allows $\beta$ to vary with $T$. 

\begin{theorem}\label{Gallagher}
Assume RH. Then, for any $\beta = \beta(T)>0$ satisfying 
\begin{equation} \label{beta condition}
\beta \, \left( \frac{\log\log T}{\log T} \right)^{\!1/2} \to 0 \quad \text{as } T\to \infty,
\end{equation}
we have
\begin{equation}\label{precise1}
\begin{split}
 \frac12 M(r_{\beta}^{-})  + \frac{1}{2}\left(1\! -\! \frac{N^*(T)}{N(T)}\right)  +o(1)   \, \leq \, \frac{N(T,\beta)}{N(T)} \, &\leq \,    \frac12 M(r_{\beta}^{+}) +  \frac{1}{2}\left(1\! -\! \frac{N^*(T)}{N(T)}\right) +o(1) 
\end{split}
 \end{equation}
 when $T$ is sufficiently large.
\end{theorem}

The condition on $\beta$ in \eqref{beta condition} arises from the size of the error term in \eqref{F alpha} below, and it may be possible to weaken this condition slightly. Since it is generally believed that the zeros of $\zeta(s)$ are all simple, we expect that $N^*(T) = N(T)$ for all $T>0$ and hence that \eqref{N star} should hold. Assuming RH, Montgomery \cite{M1} has shown that 
\begin{equation}\label{Intro_4/3_bound}
N^*(T) \le \left( \frac{4}{3} +o(1) \right) N(T) 
\end{equation}
as $T\to \infty$. Observing that $N^*(T)\ge N(T)$, and combining \eqref{precise1}, \eqref{Intro_4/3_bound}, and Theorem \ref{Intro_thm2_Gallagher_2}, we deduce the following corollary which does not rely on the additional assumption in \eqref{N star}.


\begin{corollary}\label{Mont}
Assume RH. Then, for any $\beta>0$ satisfying \eqref{beta condition}, we have
\begin{equation*}
\beta - \frac{7}{6} + \frac{1}{2 \pi^2 \beta} +  O\!\left(\frac{1}{\beta^2 }\right) +o(1)  \leq \frac{N(T,\beta)}{N(T)} \leq \beta  + \frac{1}{2 \pi^2 \beta} +  O\!\left(\frac{1}{\beta^2 }\right)  + o(1)
\end{equation*}
when $T$ is sufficiently large. 
\end{corollary}

\noindent{\sc Remark.} The lower bound in Corollary \ref{Mont} can be sharpened slightly using improved estimates for $N^*(T)$ obtained by Montgomery and Taylor  \cite{M2} (see the remark after Corollary \ref{cor_Mont_Taylor} below) or by Cheer and Goldston \cite{CG} assuming RH, or by Goldston, Gonek, \"{O}zl\"{u}k and Snyder \cite{GGOS} assuming the  generalized Riemann hypothesis for Dirichlet $L$-functions. 

\smallskip

Our original proof of Corollary \ref{Mont} was a bit different and did not rely directly on Montgomery's formula. We briefly indicate the main ideas. Writing $N(T,\beta)$ as a double sum and using a more precise formula for $N(T)$, we can show that
\begin{equation}\label{alt_proof}
N(T,\beta) \, = \, N(T) \left\{ \beta \mp \frac{1}{2}\frac{N^*(T)}{N(T)} +o(1) \right\} \ \pm \sum_{0<\gamma \le T} \!\! S\Big( \gamma \!\pm\! \frac{2\pi\beta}{\log T} \Big)
\end{equation}
for $\beta=o(\log T)$. Here, if $t$ does not correspond to an ordinate of a zero of $\zeta(s)$, we define $S(t) = \frac{1}{\pi}\arg \zeta(\frac{1}{2}+it)$ and otherwise we let
\[
S(t) = \frac{1}{2} \lim_{\varepsilon \to 0} \big\{ S(t\!+\!\varepsilon)+S(t\!-\!\varepsilon) \big\}.
\]
Using ideas from \cite{CCM}, we can replace the sum involving $S(t)$ on the right-hand side of \eqref{alt_proof} with a double sum over zeros involving the odd function $f(x) = \arctan(1/x)-x/(1+x^2)$.  In \cite{CCM}, we construct majorants and minorants of exponential type $2\pi$ for $f(x)$ using the framework for the solution of the Beurling-Selberg extremal problem given in \cite{CL} for the truncated (and odd) Gaussian. This allows us to prove the upper and lower bounds for $N(T,\beta)$ in Corollary \ref{Mont} by using these majorants and minorants in the sum on the right-hand side of \eqref{alt_proof}, twice applying the explicit formula, and then carefully estimating the resulting sums and integrals. The fact that our original proof relied on two applications of the explicit formula suggests using Montgomery's formula instead, and we have chosen only to present this simpler proof here.

\subsection{Montgomery's function $F(\alpha)$} In order to study the distribution of the differences of pairs of zeros of $\zeta(s)$, Montgomery \cite{M1} introduced the function
\begin{equation}\label{Mont_function_sec2}
F(\alpha) := F(\alpha,T) = \frac{2\pi}{T\log T}   \sum_{0< \gamma,\gamma' \le T } T^{i\alpha(\gamma'-\gamma)}\,w(\gamma'\!-\!\gamma)\,,
\end{equation}
where $\alpha$ is real, $T\ge 2$, and $w(u)=4/(4+u^2)$. Note that $F(\alpha)$ is real and that $F(\alpha)=F(-\alpha)$. Moreover, 
since
\[
 \sum_{0< \gamma,\gamma' \le T } T^{i\alpha(\gamma'-\gamma)}\,w(\gamma'\!-\!\gamma) = 2\pi \int_{-\infty}^{\infty} e^{-4\pi |u|} \Bigg| \sum_{0<\gamma\le T} T^{i\alpha\gamma} e^{2\pi i \gamma u} \Bigg|^2 \du\,,
\]
we see that $F(\alpha) \ge 0$ for $\alpha \in \mathbb{R}.$ Multiplying $F(\alpha)$ by a function $\widehat{R} \in L^1(\mathbb{R})$ and integrating, we derive the convolution formula
\begin{equation} \label{convolution}
 \sum_{0< \gamma,\gamma' \le T }  R\!\left( (\gamma'\!-\!\gamma)\frac{\log T}{2\pi} \right) w(\gamma'\!-\!\gamma) = \frac{T\log T}{2\pi} \int_{-\infty}^{\infty} \widehat{R}(\alpha) \, F(\alpha) \, \dalpha.
 \end{equation}
Assuming RH, refining the original work of Montgomery \cite{M1}, Goldston and Montgomery \cite[Lemma 8]{GM} proved that
\begin{equation} \label{F alpha}
F(\alpha) = \left( T^{-2|\alpha|} \log T + |\alpha| \right)\left(1 + O\left(\sqrt{\tfrac{\log\log T}{\log T}} \right) \right), \quad \text{as } T \to \infty,
\end{equation}
uniformly for $0\le |\alpha| \le 1$. 
Using this asymptotic formula for $F(\alpha)$ in the integral on the right-hand side of \eqref{convolution} allows for the evaluation of a large class of double sums over differences of zeros of $\zeta(s)$. 

\smallskip

From \eqref{convolution}, \eqref{F alpha}, and Plancherel's theorem, one can deduce Montgomery's formula as stated in \eqref{Mont_formula}. Furthermore, Montgomery \cite{M1} conjectured that $F(\alpha) = 1 + o(1)$ for $|\alpha| >1$, uniformly for $\alpha$ in bounded intervals. Along with \eqref{F alpha}, this conjecture completely determines the behavior of $F(\alpha)$, and suggests that Montgomery's formula in \eqref{Mont_formula} continues to hold for any function $R(x)$ whose Fourier transform $\widehat{R}(\alpha)$ is compactly supported. Choosing $R(x)$ to approximate the characteristic function of an interval led Montgomery to make the pair correlation conjecture for the zeros of $\zeta(s)$ in \eqref{PCC}.


\subsection{The Fourier transforms of $r_{\beta}^{\pm}$} Recall the entire functions $H_0(z)$ and $H_1(z)$ defined in \eqref{Intro_def_H_0} and \eqref{Intro_def_H_1}. The Fourier transform of $H_1$ is given by 
\begin{equation}\label{FTK}
\widehat{H_1}(t) = \max\big(1 \!-\! |t|, 0\big)
\end{equation}
for $t \in \mathbb{R}$, while the Fourier transform of the integrable function $W(x) = H_0(x) - \sgn(x)$ is given by \cite[Theorems 6 and 7]{V}
\begin{equation}\label{FTE}
\widehat{W}(t) = \left\{
\begin{array}{ll}
0, & {\rm if} \ \ t=0,\\
(\pi i t)^{-1} \big\{(1 - |t|)(\pi t \cot \pi t - 1)\big\}, & { \rm if} \ \ 0 < |t| <1,\\
-(\pi i t)^{-1},  &  {\rm if} \ \ |t| \geq 1.
\end{array}
\right.
\end{equation}
We can now compute the Fourier transforms of the functions $r_{\beta}^{\pm}$ defined in \eqref{Intro_BS1} and \eqref{Intro_BS2}, which, as we already noted, are continuous functions supported in $[-1,1]$. 
\begin{lemma}
For $-1 \leq t \leq 1$ we have
\begin{align}\label{FTr}
\begin{split}
\widehat{r}_{\beta}^{\pm}(t) & =  i \,\sin 2\pi \beta t\,\, \widehat{W}(t) +  \frac{\sin 2 \pi \beta t}{\pi t} \pm (1 - |t|) \cos 2\pi \beta t.
\end{split}
\end{align}
\end{lemma}
\begin{proof}
Note that
\begin{equation*}
r_{\beta}^{\pm}(x) = \frac12 \big\{ (W(x+\beta) \pm H_1(x + \beta)) + (W(-x+\beta) \pm H_1(-x + \beta))\big\} + \chi_{[-\beta, \beta]}(x).
\end{equation*}
The result now follows from \eqref{FTK} and \eqref{FTE}. 
\end{proof}

Observe from \eqref{FTr} that $\widehat{r}_{\beta}^{\pm}$ are Lipschitz functions, each with Lipschitz constant $C = O\big((1 + \beta)^2\big)$.

\subsection{Proof of Theorem \ref{Gallagher}} \label{amf} 
For any admissible function $R$, Plancherel's theorem implies that
\begin{equation} \label{Plancherel}
M(R) =  \widehat{R}(0) - \int_{-1}^1 \widehat{R}(t) \, \big(1\! -\! |t|\big)\,\dt.
\end{equation}
For simplicity, let $r_{\beta} = r_{\beta}^{\pm}$ denote either of our Beurling-Selberg functions. Then, by \eqref{N}, \eqref{convolution}, \eqref{F alpha}, \eqref{Plancherel}, and another application of Plancherel's theorem, we have
\begin{align}\label{MontFor}
\begin{split}
\frac{1}{N(T)} \sum_{0< \gamma, \gamma' \leq T} & r_{\beta} \left((\gamma' \!-\! \gamma)\tfrac{\log T}{2\pi}\right)\,w(\gamma' \!-\! \gamma) 
\\
& = \int_{-1}^{1} \widehat{r}_{\beta}(t) \left( T^{-2|t|}\, \log T + |t|\right)\dt + O\!\left( (1\!+\!\beta) \sqrt{\tfrac{\log\log T}{\log T}} \right) 
\\ 
& =  \int_{-\infty}^{\infty} \widehat{r}_{\beta}\left( \frac{u}{\log T}\right) e^{-2|u|} \,\du + \int_{-1}^{1} \widehat{r}_{\beta}(t) \, |t|\,  \dt + o(1) 
\\
&=   \widehat{r}_{\beta}(0) +  \int_{-1}^{1} \widehat{r}_{\beta}(t) \,  \dt -  \int_{-1}^{1} \widehat{r}_{\beta}(t) \, \big(1\!-\!|t|\big) \,  \dt + o(1) 
\\
&= r_{\beta}(0) +M(r_\beta) + o(1).
\end{split}
\end{align}
Here we have used the fact that 
$$|\widehat{r}_{\beta}(t)| = O(1 + \beta)$$
uniformly for all $t \in \R$, together with the assumption that $\beta$ satisfies \eqref{beta condition}, to establish the error term of $o(1)$ in \eqref{MontFor}. This error term relies, in part, on the bound (here using that $\widehat{r}_{\beta}$ has Lipschitz constant $C = O(1 + \beta)^2$), 
\begin{align*}
&\left| \int_{-\infty}^{\infty} \left\{ \widehat{r}_{\beta}\left( \frac{u}{\log T}\right) -  \widehat{r}_{\beta}(0)\right\} e^{-2|u|} \,\du \right| \leq \int_{-\infty}^{\infty} C \frac{|u|}{\log T} e^{-2|u|} \,\du = O\!\left( \frac{(1\!+\!\beta)^2}{\log T} \right)= o(1).
\end{align*}
For the majorant $r_{\beta}^{+}$, noting that $1-\frac{u^2}{4} \le w(u) \le 1$, we have
\begin{align}\label{maj11}
\begin{split}
\sum_{0< \gamma, \gamma' \leq T}  & r_{\beta}^{+} \left((\gamma' \!-\! \gamma)\tfrac{\log T}{2\pi}\right)\,w(\gamma' \!-\! \gamma) 
\\
& \geq \, r_{\beta}^{+}(0) N^*(T) \, + \!\! \sum_{\substack{0< \gamma, \gamma' \leq T\\ \gamma\neq \gamma'}} \!\!\! \chi_{[-\beta,\beta]}\! \left((\gamma' \!-\! \gamma)\tfrac{\log T}{2\pi}\right) w(\gamma' \!-\! \gamma)
\\
&= \, r_{\beta}^{+}(0) N^*(T) \, +\,  \left\{ 2 + O\!\left(\frac{(1\!+\!\beta)^2}{\log^2 T} \right) \right\} \, N(T,\beta) 
\\
&= \, r_{\beta}^{+}(0) N^*(T) \, + \, 2 N(T,\beta)  \, + \, o(T\log T),
\end{split}
\end{align}
where we have used \eqref{Fujii} and the assumption on $\beta$ in \eqref{beta condition} to estimate the error term. Using the inequalities $N^*(T) \ge N(T)$ and $r_\beta^+(0) \ge 1$, we conclude from \eqref{N}, \eqref{MontFor} and \eqref{maj11} that
\begin{equation}\label{conclusion+}
\frac{N(T,\beta)}{N(T)} \, \leq \,   \frac{1}{2} \left\{ M(r_{\beta}^{+}) + r_\beta^+(0) \left( 1 \!-\! \frac{N^*(T)}{N(T)} \right) \right\} +o(1) \, \leq \, \frac{1}{2} \left\{ M(r_{\beta}^{+}) + \left( 1 \!-\! \frac{N^*(T)}{N(T)} \right) \right\} +o(1).
\end{equation}
Similarly, for the minorant $r_{\beta}^{-}$, we obtain 
 \begin{align}\label{min11}
\begin{split}
\sum_{0< \gamma, \gamma' \leq T}  & r_{\beta}^{-} \left((\gamma' \!-\! \gamma)\tfrac{\log T}{2\pi}\right)\,w(\gamma' \!-\! \gamma) \leq  r_{\beta}^{-}(0) N^*(T) \, + \, 2 N(T,\beta)  \, + \, o(T\log T)\,,
\end{split}
\end{align}
for $\beta$ satisfying \eqref{beta condition}. In this case, since $r_\beta^-(0)\le 1$, we conclude from \eqref{N}, \eqref{MontFor} and \eqref{min11} that
\begin{equation*}
\frac{N(T,\beta)}{N(T)} \, \geq \,   \frac{1}{2} \left\{ M(r_{\beta}^{-}) + r_\beta^-(0) \left( 1 \!-\! \frac{N^*(T)}{N(T)} \right) \right\} +o(1) \, \geq \, \frac{1}{2} \left\{ M(r_{\beta}^{-}) + \left( 1 \!-\! \frac{N^*(T)}{N(T)} \right) \right\} +o(1).
\end{equation*}
This concludes the proof of Theorem \ref{Gallagher}.

\subsection{Proof of Theorem \ref{Intro_thm2_Gallagher_2}} \label{sec:evaluateMr} 

\subsubsection{Evaluation of $M(r_{\beta}^{\pm})$} We now calculate a slightly more general version of the quantity $M(r_{\beta}^{\pm})$, and specialize to the case of Theorem  \ref{Intro_thm2_Gallagher_2} at the end of this subsection. In particular, we assume the validity of Montgomery's formula in \eqref{Mont_formula} for any integrable function $R$ with Fourier transform supported in $[-\Delta,\Delta]$ with $\Delta \geq 1$ (this stronger version is used later in the proof of Theorem \ref{q theorem}).  The functions 
\begin{equation*}
s_{\Delta, \beta}^{\pm}(x) = r_{\Delta \beta}^{\pm}(\Delta x)
\end{equation*}
are a majorant and a minorant of the characteristic function of the interval $[-\beta,\beta]$ of exponential type $2\pi \Delta$, and hence with Fourier transform supported in $[-\Delta,\Delta]$. We evaluate the quantity
\begin{equation} \label{MDelta}
\frac{1}{2} M\big(s_{\Delta, \beta}^{\pm}\big) = \frac{1}{2} \widehat{s}_{\Delta, \beta}^{\pm}(0) -\frac{1}{2} \int_{-1}^{1}  \widehat{s}_{\Delta, \beta}^{\pm}(t)  (1- |t|)\,\dt,
\end{equation}
and deduce Theorem \ref{Intro_thm2_Gallagher_2} from the case $\Delta=1$.

\smallskip

First observe that
\begin{align*}
\widehat{s}_{\Delta, \beta}^{\pm}(t) & = \frac{1}{\Delta} \, \widehat{r} _{\Delta \beta}^{\pm}\left(\frac{t}{\Delta}\right)\\
& =    \frac{i}{\Delta}\,\sin 2\pi \beta t\, \widehat{W}\left(\frac{t}{\Delta}\right) +  \frac{\sin 2 \pi \beta t}{\pi t} \pm \frac{(\Delta - |t|)}{\Delta^2} \cos 2\pi \beta t\\
& =    \frac{\sin 2\pi \beta t}{\pi t} \left( 1 - \frac{|t|}{\Delta}\right)\left( \frac{\pi t}{\Delta} \cot \frac{\pi t }{\Delta} - 1\right) + \frac{\sin 2 \pi \beta t}{\pi t} \pm \frac{(\Delta - |t|)}{\Delta^2} \cos 2\pi \beta t\\
& = \frac{1}{\Delta^2} \big(\Delta - |t|\big) \sin 2 \pi \beta t \,\cot \frac{\pi t }{\Delta} + \frac{1}{\Delta} \frac{|t| \sin 2\pi \beta t }{\pi t} \pm \frac{(\Delta - |t|)}{\Delta^2} \cos 2\pi \beta t,
\end{align*}
and note that 
\begin{equation} \label{S0}
\frac{1}{2} \widehat{s}_{\Delta, \beta}^{\pm}(0)  = \beta \pm \frac{1}{2 \Delta}.
\end{equation}
Since $\widehat{s}_{\Delta, \beta}^{\pm}(t) $ is an even function, we have
\begin{align*}
\frac{1}{2} \int_{-1}^{1}   \widehat{s}_{\Delta, \beta}^{\pm}(t)  (1- |t|)\,\dt  &=  \frac{1}{\Delta^2} \int_0^1 (\Delta - t)(1-t)\sin 2 \pi \beta t \, \,\cot \frac{\pi t }{\Delta}\, \dt \\
&  \ \ \ \ \ \ \ \ \ +  \frac{1}{\Delta} \int_0^1 \frac{\sin 2\pi \beta t }{\pi} \, (1-t)\,\dt \pm \frac{1}{\Delta^2} \int_0^1 (\Delta - t)\,(1-t)\,\cos 2 \pi \beta t \, \dt\\
& := A + B \pm C,
\end{align*}
say. Integrating by parts, we find that
\begin{align} 
B  &= \frac{1}{\Delta}\left\{ \frac{1}{2 \pi^2 \beta} - \frac{\sin 2\pi \beta}{4\pi^3 \beta^2}\right\} \label{B}
\end{align}
and
\begin{align} 
C &=  \frac{1}{\Delta^2}\left\{ - \frac{(\Delta - 1)\cos 2\pi \beta}{(2 \pi \beta)^2} + \frac{(\Delta + 1)}{(2 \pi \beta)^2}  - \frac{\sin 2 \pi \beta}{4 \pi^3 \beta^3}\right\}.\label{C}
\end{align}
In order to evaluate $A$, we make use of the identity
\begin{equation*}\label{cot_identity}
i \sum_{n= -N}^N \sgn(n)\, e^{-2\pi i nt} = \cot \pi t - \left( \frac{\cos \pi (2N+1) t}{\sin \pi t}\right),
\end{equation*}
which implies that
\begin{align*}
A & =  \frac{1}{\Delta^2} \int_0^1 (\Delta - t) (1-t) \sin 2 \pi \beta t \,\left\{ i \sum_{n= -N}^N \sgn(n)\, e^{-2\pi i n\frac{t}{\Delta}}\right\} \,\dt  \\
&  \ \ \ \ \ \ \ \ \ \ \ \ \ \ \ +\frac{1}{\Delta^2} \int_0^1 (\Delta - t) (1-t) \sin 2 \pi \beta t \, \left( \frac{\cos \pi (2N+1) \frac{t}{\Delta}}{\sin \pi \frac{t}{\Delta}}\right) \,\dt\\
& := A_N + D_N,
\end{align*}
say. The Riemann-Lebesgue lemma implies that $\displaystyle{\lim_{N\to \infty} D_N = 0}$, and thus it remains to evaluate $A_N$. Interchanging summation and integration, we arrive at
\begin{align*}
A_N &= \frac{1}{\Delta^2}\sum_{n= -N}^N \sgn(n)  \int_{0}^1   \left(\frac{e^{2\pi i \beta t} - e^{- 2\pi i \beta t}}{2}\right) e^{-2\pi i n \frac{t}{\Delta}}\, (\Delta - t)(1 - t)  \,\dt  \\
& = \frac{1}{4 \pi^2} \sum_{n= -N}^N  \sgn(n)  \left\{ - \frac{(\Delta - 1)\cos 2\pi (\beta -\tfrac{n}{\Delta}) }{(\Delta \beta -n)^2} + \frac{(\Delta + 1)}{(\Delta \beta -n)^2}  - \frac{\sin 2 \pi (\beta -\tfrac{n}{\Delta}) }{\frac{\pi}{\Delta} (\Delta \beta -n)^3}\right\}.
\end{align*}
Therefore, letting $N\to \infty$, the above estimates imply that
\begin{equation}\label{FinalA}
A =  \frac{1}{4 \pi^2} \sum_{n= -\infty}^\infty  \sgn(n)  \left\{ - \frac{(\Delta - 1)\cos 2\pi (\beta -\tfrac{n}{\Delta}) }{(\Delta \beta -n)^2} + \frac{(\Delta + 1)}{(\Delta \beta -n)^2}  - \frac{\sin 2 \pi (\beta -\tfrac{n}{\Delta}) }{\frac{\pi}{\Delta} (\Delta \beta -n)^3}\right\}. 
\end{equation}

\smallskip

Combining the contributions from $A$ and $C$, we define the continuous functions $V^{\pm}: (0,\infty) \to \R$ by
\begin{equation} \label{VDelta} 
V_{\Delta}^{\pm}(\beta) =  \frac{1}{4 \pi^2} \sum_{n= -\infty}^\infty  \frac{\sgn(n^{\pm})}{(\Delta \beta -n)^2}  \left\{ - (\Delta - 1)\cos 2\pi (\beta -\tfrac{n}{\Delta}) + (\Delta + 1) - \frac{\sin 2 \pi (\beta-\tfrac{n}{\Delta}) }{\frac{\pi}{\Delta} (\Delta \beta -n)}\right\},
\end{equation}
where $\sgn(0^{\pm}) = \pm1$. Then \eqref{MDelta}, \eqref{S0}, \eqref{B}, \eqref{C}, \eqref{FinalA} and \eqref{VDelta} imply that
\begin{equation}\label{final_exp_s}
\frac{1}{2} M\big(s_{\Delta, \beta}^{\pm}\big)  =  \left(\beta \pm \frac{1}{2 \Delta}\right) - \frac{1}{\Delta}\left\{ \frac{1}{2 \pi^2 \beta} - \frac{\sin 2\pi \beta}{4\pi^3 \beta^2}\right\} - V_{\Delta}^{\pm}(\beta).
\end{equation}
Specializing to the case $\Delta = 1,$ we obtain
$$  \frac{1}{2} M\big(r_{\beta}^{\pm}\big)  =  \left(\beta \pm \frac{1}{2}\right) - \left\{ \frac{1}{2 \pi^2 \beta} - \frac{\sin 2\pi \beta}{4\pi^3 \beta^2}\right\} - V_{1}^{\pm}(\beta),$$
which is the explicit expression in Theorem \ref{Intro_thm2_Gallagher_2}.

\subsubsection{Asymptotic evaluation}  \label{sec:deduceThm1from2}  By \eqref{VDelta} we have
\begin{align} \label{eqn:asympVhaveG}
\begin{split}
V_{\Delta}^{\pm}(\beta) &= \frac{1}{4 \pi^2} \sum_{n= -\infty}^\infty  \frac{1}{(\Delta \beta -n)^2}  \left\{ - (\Delta - 1)\cos 2\pi (\beta -\tfrac{n}{\Delta}) + (\Delta + 1) - \frac{\sin 2 \pi (\beta-\tfrac{n}{\Delta}) }{\frac{\pi}{\Delta} (\Delta \beta -n)}\right\} \\
& \ \ \ \ \ \ \ \ \ \ \ - \frac{2}{4 \pi^2} \sum_{n<0 ({\rm or}\, \leq 0)}   \frac{1}{(\Delta \beta -n)^2}  \left\{ - (\Delta - 1)\cos 2\pi (\beta -\tfrac{n}{\Delta}) + (\Delta + 1) - \frac{\sin 2 \pi (\beta-\tfrac{n}{\Delta}) }{\frac{\pi}{\Delta} (\Delta \beta -n)}\right\} \\
&= \frac{1}{4 \pi^2} \sum_{n= -\infty}^\infty  \frac{1}{(\Delta \beta -n)^2}  \left\{ - (\Delta - 1)\cos 2\pi (\beta -\tfrac{n}{\Delta}) + (\Delta + 1) - \frac{\sin 2 \pi (\beta-\tfrac{n}{\Delta}) }{\frac{\pi}{\Delta} (\Delta \beta -n)}\right\} \\
& \ \ \ \ \ \ \ \ - \frac{(\Delta + 1)}{2 \pi^2 \beta \Delta} + O\!\left(\beta^{-2}\right) \\
&:=  G_\Delta(\beta) - \frac{(\Delta + 1)}{2 \pi^2 \beta \Delta} + O\!\left(\beta^{-2}\right),
\end{split}
\end{align}
say. Here we have used the estimate 
\[
\sum_{n\ge 0}  \frac{(\Delta \!-\! 1) \cos 2\pi (\beta\!+\!\tfrac{n}{\Delta})}{(\Delta \beta + n)^2} = O\!\left(\beta^{-2}\right),
\]
which follows summation by parts and the fact that
\[
\sum_{n=0}^N (\Delta\!-\!1)\cos 2\pi (\beta\!+\!\tfrac{n}{\Delta}) = O(\Delta^2)
\]
uniformly in $N$. Notice that
\begin{align*}
G_{\Delta}(\beta) = \lim_{N \rightarrow \infty} \  \frac{1}{2\Delta^2}\sum_{n= -N}^N \int_{0}^1   \left(e^{2\pi i \big(\beta - \frac{n}{\Delta}\big)t} + e^{- 2\pi i \big(\beta - \frac{n}{\Delta}\big) t}\right) \, (\Delta - t)(1 - t)  \,\dt.
\end{align*} 
Since the series defining $G_\Delta(\beta)$  in \eqref{eqn:asympVhaveG} converges uniformly for $\beta$ in a compact set, Morera's theorem can be used to show that $G_{\Delta}(\beta)$ is an analytic function of $\beta$. Thus, we can differentiate $G_{\Delta}(\beta)$ with respect to $\beta$ term-by-term, and it follows from the Riemann-Lebesgue lemma that
\begin{align*}
 G'_\Delta(\beta)  &= \lim_{N \rightarrow \infty}  \ \frac{1}{2\Delta^2}\sum_{n= -N}^N \int_{0}^1   2\pi i t \left(e^{2\pi i \big(\beta - \frac{n}{\Delta}\big)t} - e^{- 2\pi i \big(\beta - \frac{n}{\Delta}\big) t}\right) \, (\Delta - t)(1 - t)  \,\dt \\
&= \lim_{N \rightarrow \infty}  \ -\frac{1}{\Delta^2} \int_{0}^1   2\pi t \, (\Delta - t)(1 - t) \, \sin(2\pi \beta t)\, \frac{\sin (2\pi \left( N + \frac{1}{2}\right)\frac{ t}{\Delta})}{\sin \frac{\pi t}{\Delta}} \,\dt 
\\
&= 0.
\end{align*}
Therefore $G_{\Delta}(\beta)$ is a constant function in $\beta$ and, in order to determine its value, it suffices to evaluate $G_{\Delta}(0).$ 
Using the identities \cite[pp. 927--930]{MP} 
$$ \sum_{n = 1}^{\infty} \frac{\cos nx}{n^2} = \frac{1}{12} \left( 3x^2 -6\pi x +2\pi^2\right), \quad 0 \leq x \leq 2\pi,$$
and 
$$ \sum_{n = 1}^{\infty} \frac{\sin nx}{n^3} = \frac{1}{12} \left(x^3 - 3\pi x^2 + 2\pi^2x\right), \quad 0 \leq x \leq 2\pi,$$
it follows that
\begin{align*} \label{eqn:gDelta0}
G_\Delta(0)  
&= \frac{1}{2\Delta} - \frac{1}{6\Delta^2} + \frac{1}{2\pi^2}\sum_{n = 1}^\infty \left(-\frac{(\Delta - 1)\cos \frac{2\pi n}{\Delta}}{n^2} + \frac{(\Delta + 1)}{n^2} - \frac{\Delta \sin \frac{2\pi n}{\Delta}}{\pi n^3}\right) = \frac{1}{2}.
\end{align*}
Inserting this estimate into (\ref{eqn:asympVhaveG}), we derive that
\begin{equation*}\label{eqn:asymptoticofV}
V_{\Delta}^\pm(\beta) = \frac{1}{2} - \frac{(\Delta + 1)}{2\pi^2 \beta \Delta} + O\left( \beta^{-2}\right),
\end{equation*}
and therefore, from \eqref{final_exp_s}, 
\begin{equation}\label{Final_answer_M_s}
\frac{1}{2} M\big(s_{\Delta, \beta}^{\pm}\big)  =  \left(\beta - \frac{1}{2} \pm \frac{1}{2 \Delta}\right) + \frac{1}{2\pi^2\beta} + O\left(\beta^{-2}\right) .
\end{equation}
In particular, choosing $\Delta = 1,$ we deduce that
\begin{equation*}
\frac{1}{2} M\big(r_{\beta}^{\pm}\big)  =  \left(\beta - \frac{1}{2} \pm \frac{1}{2}\right) + \frac{1}{2\pi^2\beta} + O\left(\beta^{-2}\right),
\end{equation*}
and this concludes the proof of Theorem \ref{Intro_thm2_Gallagher_2}.




\section{Reproducing kernel Hilbert spaces}\label{HS_approach}

Our objective in this section is to prove Theorems \ref{Intro_HS_Thm1_RP} and \ref{Intro_HS_Thm5}. 

\subsection{Equivalence of norms via uncertainty} In order to establish the equivalence of the norms of $\mc{B}_2(\pi,\mu)$ and $\mc{B}_2(\pi)$ we shall make use of the classical uncertainty principle for the Fourier transform. The version we present here is due to Donoho and Stark \cite{DS}. 

\begin{lemma}{\rm (cf. \cite[Theorem 2]{DS})}\label{uncertainty} Let $T,W \subset \R$ be measurable sets and let $f \in L^2(\R)$ with $\|f\|_2 =1$. Then
$$|W|^{1/2}\,.\,|T|^{1/2} \geq  1 - \|f \chi_{\R\setminus T}\|_2 - \|\widehat{f} \chi_{\R\setminus W}\|_2 ,$$
where $|W|$ denotes the Lebesgue measure of the set $W$. 
\end{lemma}

\begin{lemma} \label{Lem_equiv_norms}
Let $f$ be entire. Then $f\in \mc{B}_2(\pi)$ if and only if $f\in \mc{B}_2(\pi,\mu)$.  Moreover, there exists $c>0$ independent of $f$ such that
$$c\|f\|_{2} \le \|f\|_{L^2(\dmu)} \le  \|f\|_{2}$$
for all $f\in \mc{B}_2(\pi)$. 
\end{lemma}

\begin{proof} Since $\mu$ is absolutely continuous with respect to the Lebesgue measure, it is clear that
\[
\|f\|_{L^2(\dmu)} \le \|f\|_{2}
\]
for all $f\in  \mc{B}_2(\pi)$, so in particular $\mc{B}_2(\pi)\subseteq \mc{B}_2(\pi,\mu)$. 

\smallskip

Now let $f\in\mc{B}_2(\pi,\mu)$. Since $f$ is entire, it is in particular continuous at the origin, hence $\|f\|_{2}<\infty$ and $f\in \mc{B}_2(\pi)$. It remains to show that there exists $c$, independent of $f$, with $c\|f\|_2\le  \|f\|_{L^2(\dmu)} $. We let $T = [-\tfrac18, \tfrac18]$,  $W = [-\tfrac12, \tfrac12]$ and use Lemma \ref{uncertainty} to get
$$\|f \chi_{\R\setminus T}\|_2 \geq \frac{1}{2} \|f\|_2.$$
Let $0 < \eta<1$ be such that 
$$\eta^2\, \chi_{\R\setminus T}(x) \le \left\{ 1 - \left(\frac{\sin \pi x}{\pi x}\right)^2\right\}.$$
Then
$$\frac{\eta}{2}\,  \|f\|_2 \le \eta \,\|f\chi_{\R\setminus T}\|_2 \le \|f\|_{L^2(\dmu)}.$$
This completes the proof of the lemma.
\end{proof}


\subsection{Proof of Theorem \ref{Intro_HS_Thm1_RP}} 
We start by recording the expansions:
\begin{align}\label{kw-pieces-rep}
\begin{split}
f(w,x) &=  \frac{2\pi^2  w^2}{(2\pi^2 w^2 - 1)} \int_{-\frac12}^{\frac12} e^{2\pi i x t} \, e^{-2\pi i  w t} \,\dt,\\
g(x) &= \int_{-\frac12}^{\frac12} e^{2\pi i x t} \cos\big(2^{\frac12} t\big) \,\dt,\\
h(x) &= -i \int_{-\frac12}^{\frac12}e^{2\pi i x t}  \sin \big(2^{\frac12} t\big)\,  \dt.
\end{split}
\end{align}
Define 
\begin{align*}
\kappa_w(x) &:=  f(w,x) + c(w) g(x) + d(w) h(x),\\
\ell_w(x) &:= \left\{ 1 - \left(\frac{\sin \pi x}{\pi x}\right)^2\right\}\kappa_w(x),
\end{align*}
and
$$j_w(t) := \chi_{[-\frac12,\frac12]}(t) \left\{  \frac{2\pi^2  w^2}{(2\pi^2 w^2 - 1)} e^{-2\pi i w t} + c(w) \cos\big(2^{\frac12} t\big) - i\, d(w) \, \sin\big(2^{\frac12} t\big)\right\}. $$
It follows from \eqref{kw-pieces-rep} that
\begin{align}\label{HS_eq4}
\begin{split}
\ell_w(x) &=\left\{ 1 - \left(\frac{\sin \pi x}{\pi x}\right)^2\right\}\int_{-\infty}^\infty e^{2\pi i x t} \,j_w(t) \,\dt \\
&= \int_{-\infty}^\infty e^{2\pi i x t} \left\{ j_w(t) -\int_{-1}^1 (1-|u|) \,j_w(t-u) \,\du   \right\} \dt.
\end{split}
\end{align}
Let $-\frac12 < t < \frac12$. The following identities hold for $a \in \R$:
 \begin{align*}
 \int_{t-\frac12}^{t+\frac12} \cos(a(t-u))\, \du & = \frac2a \sin(a/2),\\
 \int_{t-\frac12}^{t+\frac12} |u| \cos(a(t-u))\, \du &= \frac{2}{a^2} \cos(a/2) -\frac{2}{a^2} \cos(at) +\frac1a \sin(a/2),
 \end{align*}
and therefore
\begin{align}\label{HS_cos-int}
\begin{split}
\cos(at) \,- &\int_{-1}^{1} (1-|u|)\, \chi_{[-\frac12,\frac12]}(t-u) \cos(a(t-u))\,\du \\
&=-\frac1a \sin(a/2) + \frac{2}{a^2} \cos(a/2) + \left(1- \frac{2}{a^2}\right) \cos(at).
\end{split}
\end{align}
Similarly, we have
\begin{align}\label{HS_sin-int}
\begin{split}
\sin(at) \, - &\int_{-1}^{1} (1-|u|) \,\chi_{[-\frac12,\frac12]}(t-u) \,\sin(a(t-u))\,\du\\
&=\frac{2t}{a}\cos(a/2)  +\left(1-\frac{2}{a^2}\right) \sin(at). 
\end{split}
\end{align}
Letting $a=2^{\frac12}$ in \eqref{HS_cos-int} and \eqref{HS_sin-int} gives, for $|t| < \frac12$, the identities
\begin{align}\label{HS_eq1}
\begin{split}
\cos\big(2^{\frac12} t\big)  - \int_{-1}^1 (1-|u|)\,\chi_{[-\frac12,\frac12]}(t-u) \,\cos\big(2^{\frac12} (t-u)\big) \,\du  &= \cos\big(2^{-\frac12}\big) -2^{-\frac12} \sin\big(2^{-\frac12}\big),\\
\sin\big(2^{\frac12} t\big) - \int_{-1}^1 (1-|u|) \,\chi_{[-\frac12,\frac12]}(t-u) \,\sin\big(2^{\frac12} (t-u)\big) \,\du &= 2^{\frac12} t \cos\big(2^{-\frac12}\big) ,
\end{split}
\end{align}
while the choice $a = -2\pi  w $, for $|t|< \frac12$,  gives
\begin{align}\label{HS_eq2_0}
\begin{split}
 \frac{2\pi^2  w^2}{(2\pi^2 w^2 - 1)} &\left(e^{-2\pi i w t} - \int_{-1}^1 (1-|u|)\,\chi_{[-\frac12,\frac12]}(t-u)  \,e^{-2\pi i (t-u)w} \,\du\right) \\
 &=  e^{-2\pi i  w t} - \frac{(1-2\pi i  w t) \cos(\pi w) -\pi w\sin(\pi w)}{1-2\pi^2 w^2 }.
\end{split}
\end{align}

\smallskip

We note that $\ell_w$ has exponential type at most $3\pi$. When inserting \eqref{Intro_HS_eq3}, \eqref{HS_eq1} and \eqref{HS_eq2_0} into \eqref{HS_eq4}, the linear functions (of the variable $t$) from \eqref{HS_eq1} multiplied by $c(w)$ and $d(w)$ eliminate, for $|t|< \frac12$, the linear function in \eqref{HS_eq2_0}. Hence we obtain  
$$\ell_w(x) = \int_{-3/2}^{3/2} e^{2\pi i x t} \left(e^{-2\pi i wt} + q_w(t)\right) \dt\,,$$
where $q_w(t) =0$ for $-\frac12<t<\frac12$. Therefore
$$\ell_w(x) = \frac{\sin\pi(x-w)}{\pi(x-w)} + Q_w(x),$$
where
$$\int_{-\infty}^\infty f(x) \,Q_w(x) \,\dx =0$$ 
for all $f\in \mc{B}_2(\pi,\mu)$. This implies that 
\begin{align*}
\int_{-\infty}^\infty f(x) \, \kappa_w(x)\,\left\{ 1 - \left(\frac{\sin \pi x}{\pi x}\right)^2\right\}\dx = f(w)
\end{align*}   
for all $f\in \mc{B}_2(\pi,\mu)$. Thus $\overline{\kappa_w}$ is a reproducing kernel, and since such a kernel is unique, it follows that  $K(w,x) = \overline{\kappa_w(x)}$ as desired. This concludes the proof.

\smallskip

\noindent{\sc Remark.}  The initial guess for the reproducing kernel was found in the following way. The starting point is the function $\ell_w$ introduced in the above proof. A Fourier transform leads to the identity
\begin{equation}\label{HS_original_integral_equation}
e^{-2\pi i t w} = \widehat{\kappa}_w(t) - \int_{-1}^1 (1-|u|) \,\widehat{\kappa}_w(t-u) \,\du
\end{equation}
for $|t|<\frac12$, and two (formal) differentiations (using the fact that the second derivative of $(1-|u|)\chi_{[-1,1]}(u)$ is a linear combination of three Dirac deltas) lead to the equation
$$-4\pi^2 w^2 e^{-2\pi i t w} = \widehat{\kappa}_w''(t) - \big(\widehat{\kappa}_w(t+1) + \widehat{\kappa}_w(t-1) - 2\widehat{\kappa}_w(t)\big).$$
If $\kappa_w$ has exponential type $\pi$, then for $|t|<\frac12$ this equation simplifies to
$$
-4\pi^2 w^2 e^{-2\pi i t w} = \widehat{\kappa}_w''(t) + 2\widehat{\kappa}_w(t),$$
which can be solved explicitly. The original integral equation \eqref{HS_original_integral_equation} determines the two free parameters $c(w)$ and $d(w)$.



\subsection{A geometric lemma} Before we proceed to the proof of Theorem \ref{Intro_HS_Thm5} we present a basic lemma\footnote{A similar version of this result was independently obtained by M. Kelly and J. D. Vaaler (personal communication).} on the geometry of Hilbert spaces.

\begin{lemma}\label{HS_geometric_lemma}
Let $H$ be a Hilbert space (over $\C$) with norm $\| \cdot \|$ and inner product $\langle \cdot, \cdot \rangle$. Let $v_1, v_2 \in H$ be two nonzero vectors (not necessarily distinct) such that $\|v_1\| = \|v_2\|$ and define
$$\J = \big\{ x \in H; \ |\langle x, v_1\rangle |  \geq 1 \ {\rm and} \ |\langle x, v_2\rangle |  \geq 1\big\}.$$
Then
\begin{equation}\label{HS_Lemma_eq1}
\min_{x \in \J}\|x\| = \left(\frac{2}{\big(\|v_1\|^2 + |\langle v_1, v_2\rangle|\big)}\right)^{1/2}.
\end{equation}
The extremal vectors $y \in \J$ are given by:
\begin{enumerate}
\item[(i)] If $\langle v_1, v_2\rangle = 0$, then 
\begin{equation}\label{HS_def_x_theta_0}
y =   \left(\frac{2}{\big(\|v_1\|^2 + |\langle v_1, v_2\rangle|\big)}\right)^{1/2} \,\frac{(c_1v_1 + c_2v_2)}{\left\|v_1 + v_2\right\|}\,,
\end{equation} 
where $c_1, c_2 \in \C$ with $|c_1| = |c_2| =1$. 
\smallskip
\item[(ii)] If $\langle v_1, v_2\rangle \neq 0$, and we write $\langle v_1, v_2\rangle =  \,e^{-i\alpha} \,|\langle v_1, v_2\rangle|$, then 
\begin{equation}\label{HS_def_x_theta}
y =   \left(\frac{2}{\big(\|v_1\|^2 + |\langle v_1, v_2\rangle|\big)}\right)^{1/2}\,  \frac{c\,(e^{i\alpha } v_1 + v_2)}{\left\| e^{i\alpha} v_1 + v_2\right\|}\,,
\end{equation} 
where $c \in \C$ with $|c| =1$. 
\end{enumerate}
\end{lemma}
\begin{proof}
If $v_1$ and $v_2$ are linearly dependent the result is easy to verify, so we focus on the general case. The verification that each $y$ given by \eqref{HS_def_x_theta} belongs to $\J$ and has norm given by the right-hand side of \eqref{HS_Lemma_eq1} is straightforward. Now let 
$$\kappa := \min_{x \in \J}\|x\|\,,$$
and let $y \in \J$ be such that $\|y\| = \kappa$ (observe that such an extremal vector exists since we may restrict the search to the subspace $ {\rm span}\{v_1, v_2\}$). We consider $v_1' = e^{i\vartheta_1}v_1$ and $v_2' = e^{i\vartheta_2}v_2$ for appropriate choices of $\vartheta_1$ and $\vartheta_2$ such that
\begin{equation}\label{HS_eq2}
\langle y, v_1' \rangle \geq 1\ \  {\rm and} \ \  \langle y, v_2' \rangle \geq 1.
\end{equation} 
Since $y \in {\rm span}\{v_1', v_2'\}$, we write
$$y = a \,v_1' + b\,v_2'\,,$$
where $a,b \in \C$. The fact that $y$ satisfies \eqref{HS_eq2} implies that 
$$y' = \ov{b}\,v_1' + \ov{a}\,v_2'$$
also satisfies \eqref{HS_eq2} and thus belongs to $\J$. Therefore $z = (y + y')/2$ also satisfies \eqref{HS_eq2} and belongs to $\J$. If $y \neq y'$, from the parallelogram law we have
$$\|z\|^2  < \left\| \frac{y + y'}{2}\right\|^2 + \left\| \frac{y - y'}{2}\right\|^2= \frac{1}{2}(\|y\|^2 + \|y'\|^2) = \|y\|^2 = \kappa^2,$$
a contradiction. Therefore $b = \ov{a}$ and we have
\begin{equation}\label{HS_eq3_y_final}
y = a \,v_1' + \ov{a}\,v_2'.
\end{equation}
Having reduced our considerations to a vector $y$ of the form \eqref{HS_eq3_y_final}, we see that the two conditions in \eqref{HS_eq2} are complex conjugates, and we may work with only one of them, say $\langle y, v_1' \rangle \geq 1$. Since $\|y\| = \kappa$ is minimal, we must have the equality $\langle y, v_1' \rangle = 1$. This translates to
\begin{equation}\label{HS_key_a_eq1}
a \|v_1'\|^2 + \ov{a} \, \langle v_2', v_1' \rangle = 1,
\end{equation}
and we find
\begin{align*}
\|y\|^2 & = \left(|a|^2\|v_1'\|^2 + \ov{a}^2 \langle v_2', v_1' \rangle\right) + \left(|a|^2\|v_2'\|^2 + a^2 \langle v_1', v_2' \rangle\right)  = \ov{a} + a  = 2\, \re(a).
\end{align*}
By solving the system of equations \eqref{HS_key_a_eq1} in the variables $\re(a)$ and $\im(a)$ we arrive at
\begin{align}\label{HS_real_part}
\kappa^2 = \|y\|^2 \, = \, 2\, \re(a) \,   = \, 2 \frac{\|v_1'\|^2 - \re(\langle v_1', v_2'\rangle)}{\|v_1'\|^4 - |\langle v_1', v_2'\rangle|^2}
\,  \geq \, 2 \frac{\|v_1\|^2 - |\langle v_1, v_2\rangle|}{\|v_1\|^4 - |\langle v_1, v_2\rangle|^2}
\,  = \, \frac{2}{\|v_1'\|^2 + |\langle v_1', v_2'\rangle|},
\end{align}
and 
\begin{equation}\label{HS_im_part}
\im(a) = \frac{\im(\langle v_1', v_2'\rangle)}{\|v_1'\|^4 - |\langle v_1', v_2'\rangle|^2}.
\end{equation}
We have equality in \eqref{HS_real_part} if and only if $\langle v_1', v_2'\rangle \geq 0$. If $\langle v_1', v_2'\rangle = 0$, then $\vartheta_1$ and $\vartheta_2$ are arbitrary and $a \geq 0$. This leads to the family in \eqref{HS_def_x_theta_0}. If $\langle v_1', v_2'\rangle \neq 0$, then we must have $\vartheta_1 \equiv \vartheta_2 + \alpha \, \,({\rm mod} \,2\pi)$ and $a \geq 0$, which leads to the family in \eqref{HS_def_x_theta}.
\end{proof}

\subsection{Proof of Theorem \ref{Intro_HS_Thm5}}
Let $R$ be a nonnegative admissible function such that $R(\pm \beta) \geq 1$. Since $R$ has exponential type at most $2\pi$, by Krein's decomposition \cite[p. 154]{A} we have 
$$R(z) = S(z)\,\ov{S(\ov{z})},$$ 
where $S$ is  an entire function of exponential type at most $\pi$. On the real line we have $R(x) = |S(x)|^2$ and thus $S \in L^2(\R)$. Therefore, the function $S$ belongs to the reproducing kernel Hilbert space $\mc{H} = \mc{B}_2(\pi,\mu)$. The hypotheses imply that 
$$1 \leq |S(\beta)| =\big|\langle S, K(\beta, \cdot)\rangle_{\mc{H}} \big|$$
and
$$1 \leq |S(-\beta)| =\big|\langle S, K(-\beta, \cdot)\rangle_{\mc{H}} \big|.$$
We want to minimize the quantity
$$\|S\|^2_{\mc{H}} = \int_{-\infty}^{\infty} |S(x)|^2 \left\{ 1 - \left(\frac{\sin \pi x}{\pi x}\right)^2\right\} \,\dx.$$
By the reproducing kernel property and the symmetry of the pair correlation measure (alternatively, one can check directly by Theorem \ref{Intro_HS_Thm1_RP}), we have 
$$\|K(\beta, \cdot)\|^2_{\mc{H}} = K(\beta, \beta) = K(-\beta, -\beta) = \|K(-\beta, \cdot)\|^2_{\mc{H}}.$$ 
We are thus in position to use Lemma \ref{HS_geometric_lemma} to derive that 
\begin{align}\label{HS_asym_beta}
\|S\|^2_{\mc{H}}  \geq \frac{2}{K(\beta, \beta) + |\langle K(\beta, \cdot), K(-\beta, \cdot)\rangle_{\mc{H}}|} =  \frac{2}{ K(\beta, \beta) + |K(\beta, -\beta)|}.
\end{align}
The cases of equality in \eqref{HS_asym_beta} follow from \eqref{HS_def_x_theta_0} and \eqref{HS_def_x_theta}.

\smallskip

It remains to verify the asymptotic behavior on the right-hand side of \eqref{Intro_HS_Thm2_eq1} as $\beta \to \infty$. From Theorem \ref{Intro_HS_Thm1_RP} we get
$$K(\beta, \beta) =  \frac{2\pi^2  {\beta}^2}{(2\pi^2 {\beta}^2 - 1)} +  c(\beta) g(\beta) + d(\beta) h(\beta)$$
and 
\begin{align*}
K(\beta,-\beta) & =  \frac{2\pi^2  {\beta}^2}{(2\pi^2 {\beta}^2 - 1)} \frac{\sin2\pi \beta}{2\pi\beta}+  c(\beta) g(-\beta) + d(\beta) h(-\beta)\\
& =  \frac{2\pi^2  {\beta}^2}{(2\pi^2 {\beta}^2 - 1)} \frac{\sin2\pi \beta}{2\pi\beta}+  c(\beta) g(\beta) - d(\beta) h(\beta).
\end{align*}
Therefore, if $K(\beta,-\beta) \geq 0$ we have
\begin{align}\label{HS_Asym_1}
K(\beta,\beta) + |K(\beta,-\beta)| =  \frac{2\pi^2  {\beta}^2}{(2\pi^2 {\beta}^2 - 1)} \left(1 + \frac{\sin2\pi \beta}{2\pi\beta}\right) + 2 c(\beta) g(\beta),
\end{align} 
and if $K(\beta,-\beta) \leq 0$ we have
\begin{align}\label{HS_Asym_2}
K(\beta,\beta) + |K(\beta,-\beta)|=  \frac{2\pi^2  {\beta}^2}{(2\pi^2 {\beta}^2 - 1)} \left(1 - \frac{\sin2\pi \beta}{2\pi\beta}\right) + 2 d(\beta) h(\beta).
\end{align} 
Observe that $c(\beta) g(\pm\beta) = O(\beta^{-2})$ and that $d(\beta) h(\pm\beta) = O(\beta^{-2})$. We then have two cases to consider. First, if for large $\beta$ we have 
$$ \frac{\sin2\pi \beta}{2\pi\beta} = O\left(\beta^{-2}\right)\,,$$
then the asymptotic on the right-hand side of \eqref{Intro_HS_Thm2_eq1} is trivially true. Otherwise,   
$$\frac{\sin 2\pi \beta}{2\pi\beta} = O\big((\beta + 1)^{-1}\big).$$
Hence
$K(\beta, -\beta)$ will have the sign of  $\frac{\sin2\pi \beta}{2\pi\beta}$, and we use \eqref{HS_Asym_1} and \eqref{HS_Asym_2} to get the desired asymptotic. This concludes the proof.

\subsection{The one-delta problem} Our methods can also be used to recover the original result of Montgomery and Taylor \cite{M2} concerning the optimal majorant for the delta function with respect to the pair correlation measure. This problem was also solved, in a more general context, by Iwaniec, Luo and Sarnak \cite[Appendix A]{ILS}.

\begin{corollary}[cf. \cite{M2}]\label{cor_Mont_Taylor}
Let $R$ be a nonnegative admissible function such that  $R(0) \geq 1$. Then
\begin{align}\label{eq_Mont_Taylor}
\begin{split}
M(R) &  \geq  \frac{1}{ K(0, 0)} = 2^{-\frac12} \,\cot\big(2^{-\frac12}\big) - \frac{1}{2} = 0.3274992 \ldots
\end{split}
\end{align}
Equality in \eqref{eq_Mont_Taylor} is attained if and only if 
$$R(z) = \frac{1}{(1- 2\pi^2 z^2)^2}\left( \cos (\pi z) - 2^{\frac12}\pi z \cot \big(2^{-\frac12}\big) \sin(\pi z) \right)^2.$$
\end{corollary}
\begin{proof}
As in the proof of Theorem \ref{Intro_HS_Thm5} we may write $R(z) = S(z)\,\ov{S(\ov{z})}$, where $S \in \mc{H} = \mc{B}_2(\pi,\mu)$. Using the Cauchy-Schwartz inequality we get 
\begin{align*}
1 & \leq |S(0)|^2 =\big|\langle S, K(0, \cdot)\rangle_{\mc{H}} \big|^2 \leq \|S\|^2_{\mc{H}} \,\|K(0, \cdot)\|^2_{\mc{H}} =  \|S\|^2_{\mc{H}} \,K(0,0).
\end{align*}
Therefore, it follows that
$$\int_{-\infty}^{\infty} R(x) \left\{ 1 - \left(\frac{\sin \pi x}{\pi x}\right)^2\right\} \,\dx  = \int_{-\infty}^{\infty} |S(x)|^2 \left\{ 1 - \left(\frac{\sin \pi x}{\pi x}\right)^2\right\} \,\dx = \|S\|^2_{\mc{H}} \geq \frac{1}{K(0,0)},$$
and equality holds if and only if $S(z) = c \,K(0, z)$, where $c$ is a complex constant of absolute value $K(0,0)^{-1}$. Using the explicit representation for $K$ given by Theorem \ref{Intro_HS_Thm1_RP} we get 
\begin{align*}
R(z) &= S(z)\,\ov{S(\ov{z})}= \frac{1}{K(0,0)^2} K(0, z)^2 \\
& =  \frac{1}{K(0,0)^2}\left( \frac{1}{\big(\cos\big(2^{-\frac12}\big) - 2^{-\frac12} \sin\big(2^{-\frac12}\big)\big)}  \frac{2^{\frac12}\sin\big(2^{-\frac12}\big) \cos(\pi z) - 2\pi z \cos\big(2^{-\frac12}\big) \sin(\pi z)}{(1- 2\pi^2 z^2)}\right)^2\\
& = \left( \frac{1}{2^{\frac12}\sin\big(2^{-\frac12}\big)} \frac{2^{\frac12}\sin\big(2^{-\frac12}\big) \cos(\pi z) - 2\pi z \cos\big(2^{-\frac12}\big) \sin(\pi z)}{(1- 2\pi^2 z^2)}\right)^2\\
& = \frac{1}{(1- 2\pi^2 z^2)^2}\left( \cos (\pi z) - 2^{\frac12}\pi z \cot \big(2^{-\frac12}\big) \sin(\pi z) \right)^2.
\end{align*}
\end{proof}
\noindent {\sc Remark:} It follows from \eqref{Mont_formula}, \eqref{def-of-M}, and \eqref{eq_Mont_Taylor} that 
$$N^*(T) \leq \left(2^{-\frac12} \,\cot\big(2^{-\frac12}\big) + \frac{1}{2} + o(1) \right) N(T).$$
This inequality was previously proved by Montgomery and Taylor \cite{M2}, and can be used in the place of \eqref{Intro_4/3_bound} to give a slightly sharper version of our Corollary \ref{Mont}.


\section{Interpolation and orthogonality in de Branges spaces}\label{Sec_de_Branges_spaces}

In this section we prove Theorem \ref{Intro_thm5_super}. \new{Recall that $E$ is a Hermite-Biehler function that satisfies properties (P1) - (P4)}, and we assume without loss of generality that $E(0)>0$.

\subsection{Preliminary lemmas} We start by proving the following result.
\begin{lemma}\label{Sec15_Lem15} Let $\beta \notin \{a_k\} \cup \{b_k\}$ and consider the Hermite-Biehler function $E_{\beta}$ defined in \eqref{Intro_Def_E_beta_2}.
\begin{enumerate}
\item[(i)] The function $E_{\beta}$ satisfies properties {\rm (P1) - (P4)}.
\smallskip
\item[(ii)] If $a_k < \beta < b_{k}$, for $k \geq 1$, then $B_{\beta}(0) = B_{\beta}(\beta) = 0$.
\smallskip
\item[(iii)] If $b_k < \beta < a_{k+1}$, for $k \geq 0$, then $A_{\beta}(\beta) =0$. If $k \geq 1$, then there exists $\xi \in (0,\beta)$ such that $A_{\beta}(\xi) = 0$. 
\end{enumerate}
\end{lemma}
\begin{proof}[Proof of \textup{(i)}] Properties (P1), (P2) and (P3) are clear. A direct computation shows that
\begin{equation}\label{Sec5_Def_A_beta}
A_{\beta}(z) = \gamma_{\beta}A(z) - z B(z) 
\end{equation}
and
\begin{equation}\label{Sec5_Def_B_beta}
B_{\beta}(z) =  z A(z) + \gamma_{\beta}B(z).
\end{equation}
Suppose $A_{\beta} \in \mc{H}(E_{\beta})$. Then $A_{\beta}(x)\,|E_{\beta}(x)|^{-1} \in L^2(\R)$. Observe that
\begin{equation*}
\frac{A_{\beta}(x)}{E_{\beta}(x)} =  \gamma_{\beta} \frac{A(x)}{ (\gamma_{\beta}- ix)E(x)} - \frac{x}{(\gamma_{\beta} - ix)}\frac{B(x)}{E(x)}  = -i \frac{B(x)}{E(x)} + O(x^{-1}),
\end{equation*}
for large $x$. This would imply that $B \in \H(E)$, a contradiction. In an analogous manner, we show that $B_{\beta} \notin \mc{H}(E_{\beta})$. This establishes (P4).
\end{proof}

\begin{proof}[Proof of \textup{(ii)}] Since $E_{\beta}$ satisfies (P3), the function $B_{\beta}$ is odd and thus $B_{\beta}(0) = 0$. The fact that $B_{\beta}(\beta) = 0$ follows from \eqref{Sec5_Def_B_beta} and the definition of $\gamma_{\beta}$. 
\end{proof}

\begin{proof}[Proof of \textup{(iii)}] The fact that $A_{\beta}(\beta) = 0$ follows from \eqref{Sec5_Def_A_beta} and the definition of $\gamma_{\beta}$. Also, from \eqref{Sec5_Def_A_beta}, a number $\xi$ is a zero of $A_{\beta}$ if and only if
\begin{equation}\label{Sec5_def_xi_sol}
\xi = \gamma_{\beta} \frac{A(\xi)}{B(\xi)}.
\end{equation}
Since $B(x) >0$ in $(0,a_1]$ with $B(0) = 0$, and $A(x) >0$ in $[0, a_1)$ with $A(a_1) =0$, the function $x \mapsto A(x)/B(x)$ assumes every positive real value in the interval $(0,a_1)$. In particular, there exists $\xi \in (0,a_1)$ satisfying \eqref{Sec5_def_xi_sol}.
\end{proof}

The importance of condition (P4) lies in the fact that the sets $\{K(\xi, \cdot); \ A(\xi) = 0\}$ and $\{K(\xi, \cdot); \ B(\xi) = 0\}$ are orthogonal bases for $\H(E)$ (see \cite[Theorem 22]{B}). Using this fact, we establish four suitable quadrature formulas below. These are the key elements to prove the optimality of our approximations.

\begin{lemma}\label{Sec5_Lem16}
Let $F$ be an entire function of exponential type at most $2\tau(E)$ such that $F(x) \geq 0$ for all $x \in \R$ and 
\begin{equation}\label{Sec5_cond_M_E_finite}
M_E(F) = \int_{-\infty}^{\infty} F(x) \, |E(x)|^{-2}\,\dx < \infty.
\end{equation}
\begin{enumerate}
\item[(i)] We have 
\begin{equation}\label{Sec5_Lem16_1}
M_E(F)  = \sum_{A(\xi) = 0} \frac{F(\xi)}{K(\xi, \xi)} = \sum_{B(\xi) = 0} \frac{F(\xi)}{K(\xi, \xi)}.
\end{equation}
\item[(ii)] If $\beta \notin \{a_k\} \cup \{b_k\}$, then we have
\begin{equation}\label{Sec5_Lem16_2}
M_E(F)  = \sum_{A_{\beta}(\xi) = 0} F(\xi) \frac{\xi^2 + \gamma_{\beta}^2}{K_{\beta}(\xi, \xi)} = \sum_{B_{\beta}(\xi) = 0} F(\xi) \frac{\xi^2 + \gamma_{\beta}^2}{K_{\beta}(\xi, \xi)}.
\end{equation}
\end{enumerate}
\end{lemma}

\begin{proof}[Proof of \textup{(i)}] By \cite[Lemma 14]{CL2} (this is the corresponding version of Krein's decomposition \cite[p. 154]{A} for the de Branges space $\H(E)$) we can write $F(z) = U(z)U^*(z)$ where $U \in \H(E)$. Part (i) then follows from the orthogonal basis given by \cite[Theorem 22]{B}
\begin{align*}
M_E(F) = \int_{-\infty}^{\infty} |U(x)|^2 \, |E(x)|^{-2}\,\dx = \sum_{A(\xi) = 0} \frac{|U(\xi)|^2}{K(\xi, \xi)} = \sum_{A(\xi) = 0} \frac{F(\xi)}{K(\xi, \xi)}.
\end{align*}
A similar representation holds at the zeros of $B$.
\end{proof}
\begin{proof}[Proof of \textup{(ii)}] We now consider $F_{\beta}(z):= F(z)(z^2 + \gamma_{\beta}^2)$. This is also an entire function of exponential type at most $2\tau(E)$ which is nonnegative on the real axis. Since $|E_{\beta}(x)|^2 = |E(x)|^2 (x^2 + \gamma_{\beta}^2)$ we see from \eqref{Sec5_cond_M_E_finite} that $M_E(F) = M_{E_{\beta}}(F_{\beta}) < \infty$. We then write $F_{\beta}(z) = U_{\beta}(z)U^*_{\beta}(z)$ with $U_{\beta} \in \H(E_{\beta})$ and thus, by \cite[Theorem 22]{B}, we have
\begin{align*}
M_E(F) = M_{E_{\beta}}(F_{\beta}) = \int_{-\infty}^{\infty} |U_{\beta}(x)|^2 \, |E_{\beta}(x)|^{-2}\,\dx = \sum_{A_{\beta}(\xi) = 0} \frac{|U_{\beta}(\xi)|^2}{K_{\beta}(\xi, \xi)} = \sum_{A_{\beta}(\xi) = 0} F(\xi) \frac{\xi^2 + \gamma_{\beta}^2}{K_{\beta}(\xi, \xi)}.
\end{align*}
A similar representation holds at the zeros of $B_{\beta}$. This completes the proof of the lemma.
\end{proof}

\subsection{Proof of Theorem \ref{Intro_thm5_super}} This proof is divided in three distinct qualitative regimes.

\subsubsection{Case 1: Assume $\beta \notin \{a_k\} \cup \{b_k\} \cup (0,a_1)$} Since $E_{\beta}$ is a Hermite-Biehler function of bounded type (property (P1)), the functions $A_{\beta}$ and $B_{\beta}$ belong to the Laguerre-P\'{o}lya class (this follows from \cite[Problem 34]{B} or \cite[Lemma 13]{CL2}), i.e. they are uniform limits (in compact sets) of polynomials with only real zeros. Moreover, from property (P3) we have that $A_{\beta}^2$ and $B_{\beta}^2$ are even functions.

\smallskip

We consider now the case $b_k < \beta < a_{k+1}$ ($k \geq 1$), in which the nodes of interpolation are the zeros of $A_{\beta}$ (the proof in the case $a_k < \beta < b_{k}$ proceeds along similar lines, using the zeros of $B_{\beta}$ as interpolation nodes). Since $A_{\beta}^2$ is an even Laguerre-P\'{o}lya function with $A_{\beta}^2(\beta) = 0$ and at least two other zeros, counted with multiplicity, lie in the interval $[0,\beta)$ (by Lemma \ref{Sec15_Lem15} (iii)), the hypotheses of \cite[Theorem 3.14]{Lit} are fulfilled. Hence there exists a pair of real entire functions $R_{\beta,E}^{\pm}$ such that 
\begin{equation*}
R_{\beta,E}^-(x) \leq  \chi_{[-\beta, \beta]}(x) \leq R_{\beta,E}^+(x) 
\end{equation*}
for all $x \in \R$, with 
\begin{equation}\label{Sec5_int_1}
R_{\beta,E}^{\pm}(\xi) = \chi_{[-\beta, \beta]}(\xi)
\end{equation}
for all $\xi \neq \pm\beta$ such that $A_{\beta}(\xi) =0$, and 
\begin{equation}\label{Sec5_int_2}
R_{\beta,E}^{+}(\pm\beta)=1  \quad \text{and} \quad  R_{\beta,E}^{-}(\pm\beta)=0.
\end{equation}
Moreover, the functions $R_{\beta,E}^{\pm}$ satisfy the estimate
\begin{equation*}
\big|R_{\beta,E}^{\pm}(z)\big| \ll \frac{|A_{\beta}^2(z)|}{1 + |\re (z)|^4}
\end{equation*}
for all $z \in \C$. This shows, in particular, that the functions $R_{\beta,E}^{\pm}$ have exponential type at most $2\tau(E)$ and that $M_E(R_{\beta,E}^{\pm}) < \infty$.

\smallskip

We show next that these functions are extremal. First we consider the case of the majorant. Let $R_{\beta}^+$ be an entire function of exponential type at most $2\tau(E)$ such that 
\begin{equation*}
R_{\beta}^+(x) \geq  \chi_{[-\beta, \beta]}(x) 
\end{equation*}
for all $x \in \R$. From \eqref{Sec5_Lem16_2} we obtain
\begin{align*}
M_E(R_{\beta}^+) =  \sum_{A_{\beta}(\xi) = 0} R_{\beta}^+(\xi) \frac{\xi^2 + \gamma_{\beta}^2}{K_{\beta}(\xi, \xi)} \geq  \sum_{\stackrel{A_{\beta}(\xi) = 0}{|\xi| \leq \beta}} \frac{\xi^2 + \gamma_{\beta}^2}{K_{\beta}(\xi, \xi)},
\end{align*}
and, by \eqref{Sec5_int_1} and \eqref{Sec5_int_2}, we have equality if $R_{\beta}^+ = R_{\beta,E}^+$. The minorant case follows analogously by writing $R_{\beta}^-$ as a difference of two nonnegative functions, e.g. writing $R_{\beta}^- = R_{\beta,E}^+ - (R_{\beta,E}^+ - R_{\beta}^-)$, and applying \eqref{Sec5_Lem16_2} to each nonnegative function separately.

\subsubsection{Case 2: Assume $\beta \in \{a_k\} \cup \{b_k\}$} This case follows from the work of Holt and Vaaler \cite[Theorem 15]{HV}. In this paper they construct extremal majorants and minorants for $\sgn(x - \beta)$ that interpolate this function at the zeros of $A$ (if $\beta \in \{a_k\}$) or $B$ (if $\beta \in \{b_k\}$). If we use that
 \begin{equation*}
 \tfrac{1}{2} \big\{ \sgn (x + \beta) + \sgn(-x + \beta)\big\} = \chi_{[-\beta, \beta]}(x),
 \end{equation*}
the Holt-Vaaler construction gives us majorants and minorants for $\chi_{[-\beta, \beta]}$ that interpolate this function at the right nodes (the zeros of $A$ or $B$). The optimality now follows from \eqref{Sec5_Lem16_1} as in the previous case. We note that this is the analogous of the Beurling-Selberg construction for the Paley-Wiener case.

\subsubsection{Case 3: Assume $\beta \in (0,a_1)$} \new{A special case of this result was shown in \cite[Lemma 10]{DL} using an explicit Fourier expansion and the Cauchy-Schwarz inequality. We prove first that the zero function is an optimal minorant.} We start by noticing, from \eqref{Sec5_Def_B_beta}, that the smallest positive zero of $B_\beta$ is greater than $a_1$, since both $A$ and $B$ are positive in $(0,a_1)$. Since the zeros of $A_\beta$ and $B_\beta$ are simple and interlace, and since $A_\beta(\beta)=0$, we conclude that $\beta$ is the only zero of $A_\beta$ in the interval $(0,a_1)$. Since the zero function interpolates $\chi_{[-\beta, \beta]}$ at all the zeros of $A$ (or $A_{\beta}$), it must be an extremal function from the quadrature formula \eqref{Sec5_Lem16_1} (or \eqref{Sec5_Lem16_2}). 

\smallskip

\new{To find an extremal majorant, define the entire function $Q_\beta$ by}
\begin{equation}
Q_{\beta}(z) = C_{\beta}\frac{A_{\beta}(z)}{(\beta^2 - z^2)},
\end{equation}
where $C_{\beta}$ is a constant chosen so that $Q_{\beta}(\pm\beta) = 1$. Note that $Q_{\beta}$ is an even Laguerre-P\'{o}lya function, with no zeros in $[-\beta, \beta]$.  \new{We claim that $Q_{\beta}$ is monotone in the intervals $[-\beta,0]$ and $[0, \beta]$.} To see this, let $\pm x_{\beta}$ be the smallest zeros (in absolute value) of $Q_{\beta}$ (recall that these zeros are simple). We can then regard $Q_{\beta}$ as a uniform limit in $[-x_{\beta}, x_{\beta}]$ of even polynomials $P_k$ with only real and simple zeros. We can choose the smallest zeros (in absolute value) of $P_k$ to be $\pm x_{\beta}$. Therefore, $P_k'$ has only one simple zero in $[-x_{\beta}, x_{\beta}]$, which must be at the origin since $P_k$ is even. Moreover, by Rolle's theorem all zeros of $P_k'$ are real (and simple).  Since $P_k'(x) \to Q_{\beta}'(x)$ uniformly in $[-x_{\beta}, x_{\beta}]$, the odd function $Q_{\beta}'$ has only one zero in the interval $[-x_{\beta}, x_{\beta}]$, which must be at the origin. \new{This implies that $|Q_{\beta}|$ is monotone increasing in $[-x_{\beta},0]$ and monotone decreasing in $[0, x_{\beta}]$.  In particular, we have}
$$R_{\beta,E}^+(x):=Q_{\beta}^2(x) \geq \chi_{[-\beta, \beta]}(x)$$
for all $x \in \R$, and this is our desired majorant of exponential type at most $2 \tau(E)$. The optimality now follows from \eqref{Sec5_Lem16_2} since $R_{\beta,E}^+ =Q_{\beta}^2$ interpolates $\chi_{[-\beta, \beta]}$ at the zeros of $A_{\beta}$. 

\subsubsection{Uniqueness and relation to the two-delta problem} We have constructed above extremal functions $ R_{\beta,E}^{\pm}$ that interpolate $\chi_{[-\beta, \beta]}$ at the nodes of a certain quadrature (either $A$, $B$, $A_{\beta}$ or $B_{\beta}$). The quadrature is chosen in such a way to have $\pm \beta$ as nodes of interpolation. At these points we have $R_{\beta,E}^+(\pm \beta) = 1$ and $R_{\beta,E}^-(\pm \beta) = 0$. Therefore, the difference $R := R_{\beta,E}^+ - R_{\beta,E}^-$ is a majorant of the two-delta function $\chi_{\{\pm\beta\}}$ that interpolates $\chi_{\{\pm\beta\}}$ at the nodes of the same quadrature. From Lemma \ref{Sec5_Lem16} this difference must be an extremal function for the two-delta problem \eqref{two-Delta_E}. In particular we obtain
\begin{equation*}
\varDelta_E(\beta)  = \varLambda_E^+(\beta) - \varLambda_E^-(\beta).
\end{equation*} 
From property (P3) we have that $A$ is even and $B$ is odd. Thus, for $\beta >0$ we have
\begin{equation}\label{thm5_condition_K_beta}
K(\beta, -\beta) = \frac{B(-\beta)A(\beta)- A(-\beta)B(\beta)}{-2\pi \beta} = \frac{A(\beta)B(\beta)}{\pi \beta}.
\end{equation}
In the generic cases (iii) and (iv) we have $\beta \notin \{a_k\} \cup \{b_k\}$, and \eqref{thm5_condition_K_beta} implies that $K(\beta, -\beta) \neq 0$. In this situation, from Theorem \ref{Intro_HS_Thm5}, the extremal solution of the two-delta problem is unique, and therefore the pair of extremal functions $ R_{\beta,E}^{\pm}$ must also be unique. This concludes the proof.


\section{Small gaps between the zeros of $\zeta(s)$} \label{sg}

\smallskip

The goal of this section is to prove Theorem \ref{small gaps}. Our proof relies on the following estimate for Montgomery's function $F(\alpha) = F(\alpha, T)$ defined in \eqref{Mont_function_sec2}.

\begin{lemma} \label{Goldston ineq}
Assume RH and let $A>1$ be fixed. Then, as $T \to \infty$, we have
\[
\int_1^\xi \big( \xi-\alpha \big) \, F(\alpha) \, \dalpha \ \ge \ \frac{\xi^2}{2} - \xi + \frac{1}{3} + o(1)
\]
uniformly for $1\le \xi \le A$.
\end{lemma}

\begin{proof}
This inequality is implicit in the work of Goldston \cite[Section 7]{Gold}, but we sketch a proof for completeness. We use \eqref{convolution}, \eqref{F alpha}, and the Fourier transform pair 
\[ R_\xi(x) = \left( \frac{\sin \pi \xi x}{\pi \xi x} \right)^{\!2} \quad \text{and} \quad \widehat{R}_\xi(\alpha) = \frac{1}{\xi^2} \max\big( \xi\!-\!|\alpha|,0 \big). \]
Observe that
\begin{equation} \label{step1}
\begin{split}
1+o(1) &\le \frac{2\pi}{T\log T} \, N^*(T)
\\
& \le \frac{2 \pi}{T \log T} \sum_{0<\gamma,\gamma' \le T} R_\xi\!\left( (\gamma'\!-\!\gamma) \frac{\log T}{2\pi} \right) w(\gamma'\!-\!\gamma)
\\
&= \int_{-\xi}^\xi \widehat{R}_\xi (\alpha) F(\alpha) \, \dalpha,
\end{split}
\end{equation}
where the last step follows from the convolution formula in \eqref{convolution}. Using the fact the the integrand is even and applying \eqref{F alpha}, it follows that
\begin{equation*}
\begin{split}
\int_{-\xi}^\xi \widehat{R}_\xi (\alpha) F(\alpha) \, \dalpha & = \frac{1}{\xi} + \frac{2}{\xi^2}  \int_0^1 \big( \xi\!-\!\alpha \big)\alpha \, \dalpha + \frac{2}{\xi^2} \int_1^\xi \big( \xi\!-\!\alpha \big)  F(\alpha) \, \dalpha + o(1)
\\
& = \frac{2}{\xi} - \frac{2}{3\xi^2} + \frac{2}{\xi^2} \int_1^\xi \big( \xi\!-\!\alpha \big)  F(\alpha) \, \dalpha + o(1)
\end{split}
\end{equation*}
uniformly for $1\le \xi \le A$. Inserting this estimate into \eqref{step1} and rearranging terms, the lemma follows.
\end{proof}


\subsection{Proof of Theorem \ref{small gaps}}

We modify an argument of Goldston, Gonek, \"{O}zl\"{u}k and Snyder in \cite{GGOS} which relied on the Fourier transform pair
\[
G(x) = \left( \frac{\sin \pi x}{\pi x} \right)^2 \left(\frac{1}{1\!-\!x^2} \right)
\]
and
\[
\widehat{G}(\alpha) = \left\{
\begin{array}{cl}
1-|\alpha|+\frac{\sin 2 \pi|\alpha|}{2\pi}, & {\rm if} \ \  |\alpha|\le 1,\\
0, & { \rm if} \ \ |\alpha|>1.\\
\end{array}
\right.
\]
Note that $G(x)$ is a minorant for $\chi_{[-1,1]}$ with (nonnegative) Fourier transform supported in $[-1,1]$. Therefore, $G(x/\beta)$ is a minorant for $\chi_{[-\beta,\beta]}$, and it follows from \eqref{convolution} that
\begin{equation*}
\begin{split}
N^*(T) + 2N(T,\beta) \ &\ge \sum_{0<\gamma,\gamma' \le T} G\!\left( (\gamma'\!-\!\gamma) \frac{ \log T}{2\pi \beta} \right) w(\gamma'\!-\!\gamma)
\\
& = \left( \frac{T \log T}{2\pi} \right) \int_{-1/\beta}^{1/\beta} \beta \widehat{G}(\beta \alpha) F(\alpha) \, \dalpha.
\end{split}
\end{equation*}
Using \eqref{F alpha}, the assumption in \eqref{N star}, and the fact that the integrand is even we have
\begin{equation}\label{1th3}
N(T,\beta) \ge \left( \frac{1}{2} + o(1) \right) \frac{T \log T}{2\pi}  \left( \beta - 1 + 2\beta\int_0^1  \widehat{G}(\beta \alpha) \alpha \, \dalpha  +  2\beta\int_1^{1/\beta}  \widehat{G}(\beta \alpha) F(\alpha) \, \dalpha \right).
\end{equation}
Since $\widehat{G}(\alpha) \ge 0$ for all $\alpha$, Goldston, Gonek, \"{O}zl\"{u}k and Snyder observed that
\[
N(T,\beta) \ge \left( \frac{1}{2} + o(1) \right) \frac{T \log T}{2\pi} \left( \beta - 1 + 2\beta\int_0^1  \widehat{G}(\beta \alpha) \alpha \, \dalpha  \right),
\]
and then used a numerical calculation to show that $N(T,0.607286) \gg N(T)$ under the assumptions of Theorem \ref{small gaps}. In order to improve their result, we use Lemma \ref{Goldston ineq} to derive a lower bound for the second integral on the right-hand side of \eqref{1th3}.

\smallskip

Following Goldston \cite{Gold}, we define the function
\begin{equation*}\label{I beta}
I(\xi) = \int_1^\xi (\xi\!-\!\alpha) F(\alpha) \, \dalpha
\end{equation*}
and we observe that
\[
I'(\xi) = \int_1^\xi F(\alpha) \, \dalpha \quad \text{ and } \quad I''(\xi) = F(\xi).
\]
Note that Lemma \ref{Goldston ineq} provides a nontrivial lower bound for $I(\xi)$ as long as $\xi \ge 1 + 1/\sqrt{3}$. Integrating by parts twice, it follows that
\begin{equation}\label{2th3}
\begin{split}
\int_1^{1/\beta}  \widehat{G}(\beta \alpha) F(\alpha) \, \dalpha & = \int_1^{1/\beta}  \widehat{G}(\beta \alpha) I''(\alpha) \, \dalpha = \beta^2 \int_1^{1/\beta}  \widehat{G}''(\beta \alpha) I(\alpha) \, \dalpha.
\end{split}
\end{equation}
By definition, for $\alpha \ge 0$, we have $\widehat{G}''(\beta \alpha) = -2\pi \sin(2\pi \beta \alpha)$ which is non-negative for $1\le \alpha \le 1/\beta$ if $1/2 \le \beta \le 1$. Therefore \eqref{2th3} and Lemma \ref{Goldston ineq} imply that
\[
\int_1^{1/\beta}  \widehat{G}(\beta \alpha) F(\alpha) \, \dalpha \ge -2\pi \beta^2 \int_{1+1/\sqrt{3}}^{1/\beta}  \sin(2\pi \beta \alpha) \left(\frac{\alpha^2}{2} \!-\! \alpha \!+\! \frac{1}{3} \!+\! o(1) \right)   \dalpha
\]
for $1/2 \le \beta \le 1$. Inserting this estimate into \eqref{1th3}, it follows that
\[
N(T,\beta) \ge \left( \frac{1}{2} \!+\! o(1) \right) \frac{T \log T}{2\pi} \left( \beta \!-\! 1 \!+\! 2\beta\int_0^1  \widehat{G}(\beta \alpha) \alpha \, \dalpha  -  4\pi \beta^3\int_{1+1/\sqrt{3}}^{1/\beta}  \sin(2\pi \beta \alpha) \left(\frac{\alpha^2}{2}\! - \! \alpha \! + \! \frac{1}{3} \right)  \, \dalpha \right).
\] 
A straightforward numerical calculation shows that the right-hand side is positive if $\beta \ge 0.606894$.


\section{$q$-analogues of Theorem \ref{Intro_thm1_Gallagher} and Theorem \ref{Intro_thm2_Gallagher_2} }\label{Sec_Q_analogue}

As was suggested in Montgomery's original paper \cite{M1}, it is interesting to study the pair correlation of zeros of the family of Dirichlet $L$-functions in $q$-aspect. Montgomery had in mind improving the analogue of \eqref{F alpha} for this family of $L$-functions (see \cite{CLLR, Ozluk}), and so it is not surprising that the analogue of Theorems \ref{Intro_thm1_Gallagher} and \ref{Intro_thm2_Gallagher_2} can also be improved. In this section, we indicate such an improvement. In order to state this result, we need to introduce some notation. All sums over the zeros of Dirichlet $L$-functions are counted with multiplicity and $\varepsilon$ denotes an arbitrarily small positive constant that may vary from line to line.

\smallskip

Let $W$ be a smooth function, compactly supported in $(1,2)$. Let $\Phi$ be a function which is real and compactly supported in $(a,b)$ with $0< a< b$. As usual, define its Mellin transform $\hP$ by
$$ \hP(s) = \int_0^\infty \Phi(x)\,x^{s-1} \> \dx. $$
Suppose that $\Phi(x) = \Phi(x^{-1})$ for all $x \in \R\setminus\{0\}$  , $\hP(it) \geq 0$ for all $t \in \R$, and that $\hP(it) \ll |t|^{-2}$. For example, we may choose 
\[
\hP (s) = \left( \frac{e^{s} \!-\! e^{-s} }{2 s} \right)^2
\]
so that $ \hP ( i t)  = (\sin t/ t)^2 \geq 0$ and the function
\begin{align*}
\Phi(x) & =     \begin{cases}
    \frac12 - \frac14 \log x,  & \text{ for } 1 \leq x \leq e^2, \\
\frac12+ \frac14 \log x,   & \text{ for }  e^{-2} \leq x \leq 1, \\
0, & \text{ otherwise},
\end{cases}
\end{align*}
is real and compactly supported in $(a,b)$ for some $a,b > 0$. We define the $q$-analogue of $N(T, \beta)$ as 
\[
N_{\Phi}(Q, \beta) \, := \, \sum_{q} \frac{W(q/Q)}{\varphi(q)} {\sumstar_{\chi \,(\text{mod }{q})}} \left\{
\sum_{\substack{\gamma_{\chi}, \gamma'_{\chi} \\ 0 < \gamma_{\chi} - \gamma'_{\chi} \leq \frac{2\pi \beta}{\log Q}}} \!\!\!\!\!\!\!
\hP (i\gamma_{\chi})\hP (i\gamma'_{\chi}) \right\}.
\]
Here the superscript $\star$ indicates the sum is restricted to primitive characters $\chi \,(\text{mod }{q})$, and the inner sum on the right-hand side runs over two sets of nontrivial zeros of the Dirichlet $L$-function $L(s,\chi)$ with ordinates $\gamma_\chi$ and $\gamma'_\chi$, respectively. Similarly, we define the $q$-analogues of $N(T)$ and $N^*(T)$ by
\[
N_{\Phi}(Q) := \sum_{q} \frac{W(q/Q)}{\varphi(q)}{\sumstar_{\chi \,(\text{mod }{q})}} 
\sum_{\gamma_{\chi}}
|\hP (i\gamma_{\chi})|^2
\]
and 
\[
N_{\Phi}^*(Q) := \sum_{q} \frac{W(q/Q)}{\varphi(q)}{\sumstar_{\chi \,(\text{mod }{q})}} 
\sum_{\gamma_{\chi}}
|\hP (i\gamma_{\chi})|^2 \, m_{\gamma_\chi},
\]
respectively. Here the superscript $\star$ is as above, and $m_{\gamma_\chi}$ denotes the multiplicity of a zero of $L(s,\chi)$ with ordinate $\gamma_\chi$. Since it is generally believed that the zeros of primitive Dirichlet $L$-functions are all simple, we expect that $N^*_\Phi(Q)=N_\Phi(Q)$ for all $Q>0$. Moreover, analogous to \eqref{PCC2}, we expect that
\[
N_\Phi(Q,\beta) \sim N_{\Phi}(Q) \left\{ \beta-\frac{1}{2}+\frac{1}{2\pi^2 \beta} + O\Big(\frac{1}{\beta^2}\Big) \right\}
\]
as $\beta \to \infty$ sufficiently slowly (when $Q$ is large). In support of this, we prove the following stronger version of Theorems \ref{Intro_thm1_Gallagher} and \ref{Intro_thm2_Gallagher_2} for the zeros of primitive Dirichlet $L$-functions.

\begin{theorem} \label{q theorem} 
Assume the generalized Riemann hypothesis for Dirichlet $L$-functions, and let $\varepsilon>0$ be arbitrary. Then, for any $\beta>0$, 
we have
\[
\limsup_{Q\to\infty} \frac{N_\Phi(Q,\beta)}{N_\Phi(Q)} \le  \beta - \frac 14 + \varepsilon + \frac{1}{2 \pi^2 \beta} +  O\!\left(\frac{1}{\beta^2 }\right).
\]
If, in addition, $N_{\Phi}^*(Q) \sim N_{\Phi}(Q)$ as $Q \to \infty$, then we also have
\[
\liminf_{Q\to\infty} \frac{N_\Phi(Q,\beta)}{N_\Phi(Q)} \ge  \beta - \frac 34 -\varepsilon+ \frac{1}{2 \pi^2 \beta} +  O\!\left(\frac{1}{\beta^2 }\right).
\]
\end{theorem}

Define the $q$-analogue of Montgomery's  function $F(\alpha)$ by 
\[
F_{\Phi}(\alpha) := F_{\Phi}(\alpha,Q) = \frac{1}{N_\Phi (Q)}  \sum_{q} \frac{W(q/Q)}{\varphi(q)} {\sumstar_{\chi \,(\text{mod }{q})}}  \left| \sum_{\gamma_\chi } \hP \left( i\gamma_\chi \right)Q^{i \alpha \gamma_\chi }\right|^2.
\]
Modifying the asymptotic large sieve technique in \cite{CIS1}, Chandee, Lee, Liu and Radziwi\l\l \ 
have evaluated $F_{\Phi} (\alpha)$ when $|\alpha| < 2.$ 
\begin{lemma}\label{thm:1} \textup{(cf.  \ }\cite[Theorem 2]{CLLR}\textup{)}
Assume the generalized Riemann hypothesis for Dirichlet $L$-functions. Then, for any $\varepsilon > 0$, the estimate
\begin{align*}
  F_\Phi (\alpha) =& \, \big(1+o(1)\big) \left(  f(\alpha) + \Phi\big(Q^{-|\alpha|} \big)^2 \log Q \left( \frac{1}{ 2 \pi }  \int_{-\infty}^\infty \left| \hP ( it ) \right|^2 \dt    \right)^{-1}  \right) \\
  & \quad + O\Big(  \Phi\big( Q^{- |\alpha|} \big)  \sqrt{  f(\alpha) \log Q } \Big)
\end{align*} 
holds uniformly for $ |\alpha| \leq 2-\varepsilon$ as $ Q \to \infty$, where
$\displaystyle{  f(\alpha ) :=  \begin{cases}
| \alpha |, & \text{ for }  | \alpha |\leq 1, \\
1, & \text{ for }   | \alpha | >1 .
\end{cases}
}$
\end{lemma}

\subsection{Proof of Theorem \ref{q theorem}}

The key ingredients are Lemma \ref{thm:1} and the convolution identity
\begin{equation*} \label{plugin}
\frac{1}{N_\Phi (Q)} \sum_{q} \frac{W(q/Q)}{\varphi(q)}
\sumstar_{\chi \, \text{(mod }{q})} \sum_{\gamma_\chi, \gamma'_\chi}
R \bigg ( \frac{(\gamma_\chi \! -\! \gamma'_\chi) \log Q}{2\pi} \bigg )
\hP (i \gamma_\chi)  \hP (i \gamma'_\chi) =
\int_{-\infty}^{\infty} F_{\Phi}(\alpha) \, \widehat{R}(\alpha) \, \dalpha.
\end{equation*}
To obtain the bounds for $N_\Phi(Q, \beta)$ in Theorem \ref{q theorem}, we again use the functions 
\[
s^\pm_{\Delta, \beta}(x) = r^\pm_{\Delta\beta}(\Delta x).
\]
As stated in \textsection \ref{sec:evaluateMr}, these functions are a majorant and a minorant of $\chi_{[-\beta,\beta]}$ of exponential type $2\pi \Delta$, and thus with Fourier transform supported in $[-\Delta,\Delta]$. Using arguments similar to those in \textsection \ref{amf}, for any fixed $\beta>0$, we deduce that
\[
\limsup_{Q\to\infty} \frac{N_\Phi(Q,\beta)}{N_\Phi(Q)}   \leq \frac{1}{2} M(s^+_{\Delta, \beta}) 
\]
and, if $N_{\Phi}^*(Q) \sim N_{\Phi}(Q)$ as $Q \to \infty$, that 
\[
\liminf_{Q\to\infty} \frac{N_\Phi(Q,\beta)}{N_\Phi(Q)}  \geq \frac{1}{2} M(s^-_{\Delta, \beta}).
\]
Theorem \ref{q theorem} follows from these estimates by using \eqref{Final_answer_M_s} with $\Delta=2-\varepsilon$.


\section*{Appendix A}\label{App}
Here we prove the following result that was left open in the introduction.
\begin{proposition}
The entire function $E(z)$ defined in \eqref{Intro_Def_E_special} satisfies properties \eqref{Intro_HB_cond} and \eqref{Intro_rep_kernel}.
\end{proposition}
\begin{proof}
 We start by observing that $K(w,z)$ defined by Theorem \ref{Intro_HS_Thm1_RP} verifies the following properties:
\begin{enumerate}
\item[(i)] $K(w,w) > 0$ for all $w \in \C$. In fact, if we had $K(w,w)=0$ for some $w \in \C$, this would imply that $f(w) = 0$ for every $f \in \H$. This is a contradiction.
\smallskip
\item[(ii)] $K(\ov{w},z) = \ov{K(w, \ov{z})}$ for all $w,z \in \C$. This is a direct verification.
\smallskip
\item[(iii)] $K(w,z) = \ov{K(z,w)}$. This follows from the reproducing kernel property:
\begin{align*}
K(w,z) = \langle K(w,\cdot), K(z, \cdot) \rangle_{\H} = \ov{\langle K(z,\cdot), K(w, \cdot) \rangle_{\H}} = \ov{K(z,w)}.
\end{align*}
\end{enumerate}
Whenever $f \in \H$ has a nonreal zero $w$, the function $z\mapsto f(z)(z-\ov{w})/(z-w)$ belongs to $\H$ and has the same norm as $f$. From the first part of the proof of \cite[Theorem 23]{B} we have that $L(w,z) = 2\pi i (\overline{w} - z) K(w,z)$ satisfies the identity
\begin{equation}\label{App_1}
L(w,z) = \frac{L(\alpha, z)L(w, \alpha)}{L(\alpha, \alpha)} + \frac{L(\ov{\alpha},z) L(w, \ov{\alpha})}{L(\ov{\alpha}, \ov{\alpha})}
\end{equation}
for all nonreal $\alpha$ and $w,z \in \C$. From property (ii) above and the fact that $K(\ov{\alpha}, \ov{\alpha})$ is real, we get that 
\begin{equation}\label{App_2}
L(\alpha, \alpha) = - L(\ov{\alpha}, \ov{\alpha}).
\end{equation}
Taking $w = z$ in \eqref{App_1}, and using \eqref{App_2} and properties (ii) and (iii) above, we get
\begin{equation}\label{App_3}
L(z,z) = \frac{|L(\alpha, z)|^2}{L(\alpha, \alpha)} - \frac{|L(\alpha, \ov{z})|^2}{L(\alpha, \alpha)}.
\end{equation}
If we consider any (fixed) $\alpha \in \C^{+}$, we have $L(\alpha, \alpha) >0$, and we can define the entire function (of the variable $z$)
\begin{equation*}
E(\alpha,z) :=  \frac{L(\alpha,z)}{L(\alpha, \alpha)^{\frac12}}.
\end{equation*}
From \eqref{App_3} and property (i), we find that
\begin{equation*}
|E(\alpha, z)|^2 - |E(\alpha,\ov{z})|^2 = L(z,z) = 4\pi\,\im(z)\, K(z,z) >0
\end{equation*}
for all $z \in \C^+$. This is the Hermite-Biehler property. The identity
\begin{equation*}
L(w,z) = E(\alpha, z) \ov{E(\alpha, w)} -  \overline{E(\alpha, \overline{z})}E(\alpha, \ov{w})
\end{equation*}
is equivalent to \eqref{App_1}. In our particular case, we simply choose $\alpha = i$.
\end{proof}


\section*{Acknowledgements.} 
\noindent EC acknowledges support from CNPq-Brazil grant $302809/2011-2$, and FAPERJ grant $E-26/103.010/2012$. MBM is supported in part by an AMS-Simons Travel Grant and the NSA Young Investigator Grant H98230-13-1-0217. We would like to thank IMPA -- Rio de Janeiro and CRM -- Montr\'{e}al for sponsoring research visits during the development of this project.





\linespread{1.2}

\end{document}